\documentclass{pspum-l}
\usepackage{dsfont,hyperref,amssymb}

\newcommand{\cA}{\mathcal{A}}

\newcommand{\cC}{\mathcal{C}}
\newcommand{\cD}{\mathcal{D}}
\newcommand{\cE}{\mathcal{E}}
\newcommand{\cF}{\mathcal{F}}
\newcommand{\cG}{\mathcal{G}}
\newcommand{\cH}{\mathcal{H}}

\newcommand{\cK}{\mathcal{K}}

\newcommand{\cM}{\mathcal{M}}
\newcommand{\cN}{\mathcal{N}}
\newcommand{\cO}{\mathcal{O}}

\newcommand{\cQ}{\mathcal{Q}}
\newcommand{\cR}{\mathcal{R}}

\newcommand{\cT}{\mathcal{T}}
\newcommand{\cU}{\mathcal{U}}
\newcommand{\cV}{\mathcal{V}}
\newcommand{\cW}{\mathcal{W}}
\newcommand{\cX}{\mathcal{X}}

\newcommand{\dC}{{\mathds C}}
\newcommand{\dQ}{{\mathds Q}}
\newcommand{\dN}{{\mathds N}}

\newcommand{\dR}{{\mathds R}}

\newcommand{\dZ}{{\mathds Z}}
\newcommand{\dP}{{\mathds P}}

\newcommand{\dA}{{\mathbb A}}

\DeclareMathOperator{\DcPMHS}{\check{\mathit{D}}_{\mathit{PMHS}}}
\DeclareMathOperator{\DPMHS}{\mathit{D}_{\mathit{PMHS}}}
\DeclareMathOperator{\DcPHS}{\check{\mathit{D}}_{\mathit{PHS}}}
\DeclareMathOperator{\DPHS}{\mathit{D}_{\mathit{PHS}}}
\newcommand{\cbl}{{\check{D}_{BL}}}
\newcommand{\cblp}{\check{D}_{BL}^{pp}}

\newcommand{\MBL}{\cM_{BL}}
\newcommand{\MBLp}{\cM^{pp}_{BL}}
\newcommand{\TERP}{(H,H'_\dR,\nabla,P)}
\newcommand{\Gr}{\mathit{Gr}}
\DeclareMathOperator{\Sp}{\textup{Sp}}
\DeclareMathOperator{\rk}{\textup{rank}}
\DeclareMathOperator{\Spp}{\textup{Spp}}
\DeclareMathOperator{\Tr}{\textup{Tr}}
\DeclareMathOperator{\Aut}{\textup{Aut}}
\DeclareMathOperator{\diag}{\textup{diag}}

\newtheorem{theorem}{Theorem}[section]
\newtheorem{lemma}[theorem]{Lemma}
\newtheorem{proposition}[theorem]{Proposition}
\newtheorem{corollary}[theorem]{Corollary}

\theoremstyle{definition}
\newtheorem{definition}[theorem]{Definition}
\newtheorem{example}[theorem]{Example}

\theoremstyle{remark}
\newtheorem{remark}[theorem]{Remark}

\numberwithin{equation}{section}

\begin{document}

\title{Twistor structures, $tt^*$-geometry and singularity theory}

\author{Claus Hertling}
\address{Lehrstuhl VI f\"ur Mathematik, Universit\"at Mannheim, A5 6, 68131 Mannheim, Germany}
\email{Claus.Hertling@uni-mannheim.de}

\author{Christian Sevenheck}
\address{Lehrstuhl VI f\"ur Mathematik, Universit\"at Mannheim, A5 6, 68131 Mannheim, Germany}
\email{Christian.Sevenheck@uni-mannheim.de}

\date{October 19, 2007.}



\begin{abstract}
We give an overview on the $tt^*$-geometry defined for isolated hypersurface
singularities and tame functions via Brieskorn lattices. We discuss nilpotent
orbits in this context, as well as classifying spaces of Brieskorn lattices
and (limits of) period maps.
\end{abstract}

\maketitle

\section{Introduction}
\label{secIntro}

$tt^*$-geometry is a generalization of variation of Hodge structures.
It appeared first in papers of Cecotti and Vafa (\cite{CV1,CV2}).
Their work connects to singularity theory through Landau-Ginzburg models. Here
the $tt^*$-geometry is defined by oscillating integrals of
tame algebraic functions and a careful handling of the underlying real structure,
which leads to antiholomorphic data. A formalization of this construction was given in \cite{He4},
where the basic object is a twistor, i.e., a vector bundle on $\dP^1$, with some additional structure, namely,
a flat connection with poles of order 2 at 0 and $\infty$. Such a twistor is is a
holomorphic object, however, a family (or variation) of it over some parameter space is not.

The tame functions in Landau-Ginzburg models have isolated singularities, and the local geometry
of these singularities has been studied for almost 40 years.
The transcendental data of these germs of functions, i.e., Gau{\ss}-Manin connection, Brieskorn lattice
and the polarized mixed Hodge structure of Steenbrink and Varchenko
fit very well into the new framework provided by $tt^*$-geometry.
In the case of germs of functions, it is possible to imitate oscillating integrals by a
Fourier-Laplace transformation. The holomorphic theory of oscillating integrals for
global tame functions had been studied by Pham and, more recently, by Sabbah  and Douai
(\cite{Sa2,DS1}). Sabbah also recovered a basic positivity result
of Cecotti and Vafa for that case (\cite{Sa8}, see also Theorem \ref{theoTamePurePolarized}).

A second, even older starting point is the work of Simpson on harmonic
bundles and twistors (\cite{Si1,Si5}), although the notion of harmonic bundles is
slightly weaker. His work has been continued and generalized by Sabbah
(\cite{Sa6}) and Mochizuki (\cite{Mo2}) in Sabbah's notion of
polarized twistor $\cD$-modules. Mochizuki proved results which generalize to a large extent
Schmid's work on variations of Hodge structures, nilpotent orbits and limit mixed Hodge structures.

The theory of singularities of local and global functions lies at the crossroad of
all these developments. In this survey we concentrate on the local functions
(this is referred to as the regular singular case later on). We show how the
work of Schmid and Mochizuki applies, we discuss nilpotent orbits,
classifying spaces and (limits of) period maps.

The main references for the material covered here are the articles \cite{He4,HS1,HS2} and
\cite{HS3}. We will give precise definitions and complete statements for all
the results presented, but almost no proofs for which we refer to the
above papers.

Here is a short outline of the contents of this survey.
In the next two sections, we recall classical material: the basic definition
and properties of Brieskorn lattices of isolated hypersurface singularities
and tame functions on affine varieties, their Fourier-Laplace transformation, and
how to construct (variation of) twistor structures from them (along with the definition
of the latter). An axiomatic framework for the various types of Brieskorn lattices
is given under the name TERP-structure (an abbreviation for ``twistor'', ``extension'',
``real structure'' and ``pairing'').
Section \ref{secNilpotent} discusses the relation between (polarized) twistor
structures and (polarized mixed) Hodge structures defined by filtrations associated
to a Brieskorn lattice. The next two sections (\ref{secClassifying} and \ref{secCompact})
investigate the twistor geometry on classifying spaces of Brieskorn lattices
(resp. TERP-structures).
These spaces were constructed in order to study Torelli theorems
for so-called $\mu$-constant families of singularities. We give a result on the curvature of the
Hermitian metric
induced on them. Moreover, we construct partial compactifications of these
spaces on which the above metrics are complete. This allows us to deduce some results
known for variations of Hodge structures in the case of TERP/twistor-structures using
standard techniques from complex hyperbolic analysis.
Finally, the last section discusses period mappings for hypersurface
singularities, mainly for families over a punctured disc, and their behaviour at the boundary point $0\in\Delta$.
We show some consequences of Mochizuki's construction of a limit polarized mixed twistor
and give statements analogues to those obtained for variations of Hodge structures.

\section{Brieskorn lattices}
\label{secBrieskorn}

We will give here a very brief reminder on the theory of Brieskorn lattices
arising from isolated hypersurface singularities or polynomial functions. The
goal is to motivate the constructions in the later sections, and to
provide the main class of examples.

Consider first the classical situation of a germ of a holomorphic function
$f:(\dC^{n+1},0)\rightarrow (\dC,0)$ with an isolated critical point. By choosing appropriate
representatives $X$ and $S$ for $(\dC^{n+1},0)$ and $(\dC,0)$, this function germ yields
a mapping $f:X\rightarrow S$ with the property that its restriction to $X':=f^{-1}(S')$,
where $S':=S\backslash\{0\}$, is a locally trivial fibration, called the
Milnor fibration. Its fibres are homotopic to a bouquet of
$\mu$ spheres of (real) dimension $n$, where
$\mu:=\dim_\dC(\cO_{\dC^{n+1},0}/J_f)$
is the Milnor number; here
$J_f:=(\partial_{x_0}f,\ldots,\partial_{x_{n}}f)$. In particular, $R^n\!f_*\dC_{X'}$ is a local system,
in other words, the sheaf $\cF:=R^n\!f_*\dC_{X'}\otimes_{\dC_{S'}}\cO_{S'}$ defines a holomorphic bundle
with a flat connection $\nabla$. It was Brieskorn's idea to consider extensions of
$(\cF,\nabla)$ over the whole of $S$.
\begin{theorem}[\cite{Bri,Mal0,Greu1}]
Consider the following three sheaves and mappings between them:
$$
\cH^{(-2)} := \cH^n(f_*\Omega^\bullet_{X/S}) \stackrel{\alpha}{\hookrightarrow} \cH^{(-1)}:=\frac{f_*\Omega^n_{X/S}}{df\wedge f_*\Omega^{n-1}_{X/S}}
\stackrel{\beta}{\hookrightarrow} \cH^{(0)}:=\frac{f_*\Omega^{n+1}_{X}}{df\wedge df_*\Omega^{n-1}_{X}}
$$
(Note that these are not Brieskorn's original notations).
Here $\alpha$ is the natural inclusion and $\beta(\widetilde{\omega}):=df\wedge\widetilde{\omega}$.
All of them are $\cO_S$-locally free and restrict to $\cF$ on $S'$.
The quotients ${\cC}\!oker(\alpha)$ and ${\cC}\!oker(\beta)$ are isomorphic to $f_*\Omega^{n+1}_{X/S}$
($\beta:{\cC}\!oker(\alpha)\stackrel{\cong}{\rightarrow}f_*\Omega^{n+1}_{X/S}$ ), in particular,
supported on $0\in S$ and of dimension $\mu$. Define differential operators
$$
\begin{array}{rclcrcl}
\nabla_t:\cH^{(-2)}_0 & \longrightarrow & \cH^{(-1)}_0 & \hspace*{1cm} & \nabla_t:\cH^{(-1)}_0 & \longrightarrow & \cH^{(0)}_0 \\ \\
\omega & \longmapsto & \eta & \hspace*{1cm} & \widetilde{\omega} & \longmapsto & d\widetilde{\omega}
\end{array}
$$
where $\eta$ is such that $d\omega=df\wedge \eta$, and
where $t$ is the coordinate on $S$. These give an extension of the topological connection
$\nabla$ on $\cF$ and define meromorphic connections with a regular
singularity at $0\in S$ on $\cH^{(-2)}$, $\cH^{(-1)}$ and $\cH^{(0)}$,
because  $t^{-k}\cH^{(-2)}\supset \cH^{(-1)}$ and
$t^{-k}\cH^{(-1)}\supset\cH^{(0)}$
where $k$ is minimal with $f^k\in J_f$.
Moreover, both are isomorphisms of $\dC$-vector spaces.
In particular, the inverse $\partial^{-1}_t:=\beta\circ\nabla^{-1}_t$
is well-defined on $\cH^{(0)}_0$ and moreover, the germ
$\cH^{(0)}_0$ is free of rank $\mu$ over the ring
$\dC\{\{\partial^{-1}_t\}\}$ of microdifferential operators with constant
coefficients.
\end{theorem}
The most interesting of the three extensions is $\cH^{(0)}$, mainly because
$\Omega^{n+1}_X$ is a line bundle, so for fixed coordinates, we may represent
elements in $\cH^{(0)}$ by (classes of) functions.
$\cH^{(0)}$ is called the \textbf{Brieskorn lattice} of the function germ $(f,0)$.

In most of the applications, one is interested in the case of families of function germs.
\begin{theorem}[\cite{Greu1}]\label{defBrieskornFamily}
Let $F:(\dC^{n+1}\times M,0)\rightarrow (\dC,0)$ be a germ where $M\subset \dC^m$
is open such that $f:=F_{|{\dC^{n+1}\times\{0\}}}$ has an isolated singularity as before, then
the relative Brieskorn lattice
$$
\cH^{(0)}:=\frac{\varphi_*\Omega^{n+1}_{\cX/M}}{dF\wedge d\varphi_*\Omega^{n+1}_{\cX/M}}
$$
(where $F:\cX\rightarrow S$ is a good representative as before, and
$\varphi:=F\times\mathit{pr}_M:\cX\rightarrow S\times M$) is $\cO_{S\times M}$-locally
free and carries an integrable connection, which is meromorphic along the discriminant
$$
\cD:=\left\{(t,\underline{y})\,|\,\varphi^{-1}(t,\underline{y})\mbox{ singular }\right\}\subset S\times M
$$
It is defined by $\nabla_X(\omega):=Lie_{\widetilde{X}}(\omega)-\partial_t(Lie_{\widetilde{X}}(F)\omega)$, where
$(\widetilde{X},X)\in\cT_\cX\times\cT_{S\times M}$
such that $d\varphi_*(\widetilde{X})=X$ (in particular, $X\in \Theta_{S\times M}(\log \cD)$).
\end{theorem}
For a given germ $(f,0)$, there are two particularly interesting cases of such deformations, namely:
\begin{itemize}
\item
Take $M$ to be an open ball in $\dC^\mu$, and put $F:=f+\sum_{i=1}^\mu y_i g_i$, with
$(g_1,\ldots,g_\mu)\in (\dC\{x\})^\mu$ representing a basis of $\dC\{x\}/J_f$.
Then $F$ is called the semi-universal unfolding of $f$.
The pair $(\cH^{(0)},\nabla)$ has a logarithmic pole along $\cD$ in this case.
\item
Take any family $F$ as above with the property that for each $\underline{y}\in M$, the only critical value of
$F_{\underline{y}}$ is zero. A result of Gabrielov, Lazzeri and L\^{e} shows
that then the only critical point of $F_{\underline{y}}$ is the origin
of $\dC^{n+1}$. These families are called ``$\mu$-constant deformations'' of $f$.
\end{itemize}

The second class of objects we are interested in is closely related to these function germs,
but in this case phenomena of a more global nature occur. We start
with a smooth affine manifold $X$ and we want to study regular functions $f:X\rightarrow \dA^1_\dC$. There
is a condition one needs to impose in order to make sure that no change of topology is caused by singularities
at infinity (of a partial compactification $\overline{f}:\overline{X}\rightarrow \dA^1_\dC$).
\begin{definition}[\cite{NZ}]
Let $f:X\rightarrow \dA^1_\dC$ be a regular function; then $f$ is called M-tame iff
for some embedding $X\subset \dA^N$ and some $a\in \dA^N$, we have that for any $\eta>0$
there is some $R(\eta)>0$ such that for any $r\geq R(\eta)$, the spheres $\|x-a\|^2=r$ are transversal
to all fibres $f^{-1}(t)$ for $|t|<\eta$; here $\|\cdot\|$ denotes the Euclidean distance
on $\dC^N$.
\end{definition}
For any regular function $f$ as above, we define
$$
M_0:=\frac{\Omega^{alg,n+1}_X}{df\wedge d\Omega^{alg,n-1}_X}
$$
(where $\Omega^{alg,k}$ are the
algebraic differential $k$-forms on $X$)
to be the algebraic Brieskorn lattice of $f$. The basic structure result (\cite{NS,Sa2}) is the
following.
\begin{theorem}
If $f$ is M-tame, then $M_0$ is a free $\dC[t]$-module.
\end{theorem}
Note that the rank of $M_0$ is not
necessarily equal to
$\mu:=\sum_{x\in X}\mu(X,x)$, where $\mu(X,x)$ is the local Milnor number
of the germ $f:(X,x)\rightarrow (\dC,f(x))$. It is so if the space $X$ is
contractible, e.g., $X=\dA^{n+1}$.

The next step is to define the Fourier-Laplace transformation of the Brieskorn lattice.
In fact, the proper framework to carry this out requires the study of the Gau{\ss}-Manin
system, which is an (algebraic or analytic) $\cD$-module, in which the Briekorn lattice is embedded.
One first defines the Fourier-Laplace transformation of this object and has to analyze
what happens with the Brieskorn lattice under this operation. Details can be found
in \cite{Sa2} or \cite{He4}.

We will need here and in the sequel the notion of the Deligne extensions of a flat bundle over a divisor.
More precisely, let $H'$ be a flat bundle on $\dC^*$ (let $t$ be a coordinate on $\dC$), write
$H^\infty=\oplus_{\lambda\in\dC}H^\infty_\lambda$ for the generalized eigen-decomposition of the space of flat sections with respect to
the monodromy, then we denote by $V^\alpha$ (resp. $V^{>\alpha}$) the locally free extension of $H'$ to $\dC$ generated
by (the so-called elementary) sections $t^{\beta-\frac{N}{2\pi i}} A$, where $N$ is the logarithm of
the unipotent part of the monodromy of $H'$, $A\in H_{e^{-2\pi i \beta}}^\infty$
and $\beta\geq\alpha$ (resp. $\beta>\alpha$). We denote similarly by $V_\alpha$
(resp. $V_{<\alpha}$)
the corresponding Deligne extensions of $H'$ over infinity. Finally, let $V^{>-\infty}:=\cup_{\alpha}V^\alpha$
(resp. $V_{<\infty}:=\cup_{\alpha}V_\alpha$) be the meromorphic extensions consisting
of sections of $H'$ with moderate growth at zero, resp. infinity.

\begin{theorem}[\cite{Sa2,He4}]
For a local singularity $f:(\dC^{n+1},0)\rightarrow (\dC,0)$, consider
$M_0:=H^0(\dP^1, l_* V_{<\infty}\cap\widetilde{i}_*\cH^{(0)})$, where
$\widetilde{i}:\dC\hookrightarrow \dP^1$ and $l:\dP^1\backslash\{0\}\hookrightarrow\dP^1$.
Then $M_0$ is a free $\dC[z]$-module, as is the Brieskorn lattice $M_0$ defined
for a tame function. In both cases, define the following operators:
$$
\tau := \nabla_t
\quad\quad\quad
\mbox{and}
\quad\quad\quad
\nabla_\tau := -t
$$
In the local case, denote by $G_0$ the vector space $M_0$, seen as a $\dC[\tau^{-1}]$-module.
If $f$ is a tame function, define $G_0:=M_0[\partial^{-1}_t]$. Then in both cases
$G_0$ is a free $\dC[\tau^{-1}]$-module, invariant under $-\nabla_\tau$, i.e.,
a free $\dC[z]$-module, invariant under $z^2\nabla_z$, where $z:=\tau^{-1}$.
In both cases it has rank $\mu$, in contrast to $M_0$ in the global case.
\end{theorem}
We will mainly work with analytic objects, therefore we consider
the holomorphic vector bundle over $\dC$ corresponding to $G_0$, denoted by $H$
in the sequel. By definition it is equipped with a connection with a pole of order at most two at zero.
Its restriction to $\dC^*$ is then necessarily flat; we denote this
restriction by $H'$ (note that in general, $H'$ is not equal to
the local system $F$ from the beginning of this section, although they are related,
as we will see).

We are going to describe briefly another approach to the bundle $H$ which is of more topological
and analytical nature. For simplicity, we restrict to the case of tame functions,
the local case is treated in detail in \cite[Ch. 8]{He4}.
We start with a direct construction of the flat bundle $H'$, then we outline how
to obtain the extension $H$. Let as above $f:X\rightarrow \dA^1_\dC$ be tame, and suppose
moreover that all critical points of $f$ are non-degenerate
(although this seems to be a severe restriction, it is relatively easy to
show that the construction for the general case can always be reduced to the one
described by a small deformation, called Morsification).
Choose an embedding $X\subset \dA^N$ and $\eta>0$ such that all
critical values of $f$ are contained in $S:=\{t\ |\ |t|<\eta\}$, and
choose $R(\eta)>0$ sufficiently large.
Then $U:=\{x\in X\ |\ \|x\| <R(\eta), f(x)\in S\}$ contains all critical points of $f$.
For any $z\in\dC^*$, choose non-intersecting paths $\gamma_i$ inside $S\backslash\{f(\mathit{Crit}(f))\}$
from $\eta\frac{z}{|z|}$ to the critical values of $f$, and construct for each such path
a continuous family $\Gamma_i$ (called the \emph{Lefschetz thimble}) of vanishing
cycles in $f^{-1}(\gamma_i)$.
Then we have that $\Lambda_z:=H_n(U,f^{-1}(\eta\frac{z}{|z|}),\dZ) \cong \oplus_i \dZ [\Gamma_i]$
and we put $H'_z:={\rm Hom}_\dZ(\Lambda_z,\dC)$, which defines a local system of rank $\mu$ on $\dC^*$.
Denote by $\cH'$ its sheaf of holomorphic sections, then one can prove that the
intersection of Lefschetz thimbles on opposite fibres is well-defined (as it takes
place in a compact subset of $U$), which gives a perfect pairing $\Lambda_z\times\Lambda_{-z}\rightarrow\dZ$,
and induces a flat pairing
$$
P:\cH'\otimes j^*\cH'\longrightarrow \cO_{\dC^*}
$$
Define a map
$$
\begin{array}{rcl}
{\rm osc}:\Omega^{alg,n+1}_X & \longrightarrow & i_*\cH' \\ \\
\omega & \longmapsto & [\underbrace{z}_{\in\dC^*}\mapsto
(\underbrace{[\Gamma_i]}_{\in \Lambda_z}\mapsto
\underbrace{\int_{\widetilde{\Gamma}_i}e^{-\frac{t}{z}}\omega}_{\in\dC})]
\end{array}
$$
Here $i:\dC^*\hookrightarrow \dC$ is the inclusion and $\widetilde{\Gamma}_i$
denotes a family of vanishing cycles in $f^{-1}(\widetilde{\gamma}_i)$ extending
$\Gamma_i$, where $\widetilde{\gamma}_i$ is an extension of $\gamma_i$ to a path from a critical value
to infinity, with asymptotic direction $\frac{z}{|z|}$.
It follows from work of Pham (\cite{Ph3,Ph4}) that these integrals are well-defined.
Put $\cH:={\rm Im}({\rm osc})$, this gives
a vector bundle over $\dC$, and one checks easily that the connection has a pole of order at most two.

The following picture summarizes and illustrates the geometry
used above in the definition of $\cH'$, $P$ and $\cH$.
\begin{center}
\setlength{\unitlength}{0.8cm}
\noindent
\begin{picture}(12,12)

\bezier{300}(2,10)(2,11)(6,11)
\bezier{300}(2,10)(2,9)(6,9)
\bezier{300}(10,10),(10,11),(6,11)
\bezier{300}(10,10),(10,9),(6,9)

\put(2,10){\line(0,-1){5.5}}
\put(10,10){\line(0,-1){5.5}}

\bezier{300}(2,4.5)(2,5.5)(6,5.5)
\bezier{300}(2,4.5)(2,3.5)(6,3.5)
\bezier{300}(10,4.5),(10,5.5),(6,5.5)
\bezier{300}(10,4.5),(10,3.5),(6,3.5)

\bezier{300}(2,2)(2,3)(6,3)
\bezier{300}(2,2)(2,1)(6,1)
\bezier{300}(10,2),(10,3),(6,3)
\bezier{300}(10,2),(10,1),(6,1)

\put(0.5,5){\makebox(0,0)[bl]{$U$}}
\put(0.8,4.5){\vector(0,-1){1.5}}
\put(0.5,2){\makebox(0,0)[bl]{$S$}}
\put(1,3.5){\makebox(0,0)[bl]{$f$}}

\put(8,9.1){\line(0,-1){5.5}}
\bezier{50}(9,10.75),(9,8)(9,5.25)

\bezier{200}(8,9.1)(8.3,10)(9,10.75)
\bezier{20}(8,3.6)(8.3,4.5)(9,5.25)
\bezier{200}(8,1.1)(8.3,2)(9,2.75)

\bezier{200}(4,8)(4.3,8)(4.6,8.5)
\bezier{200}(4,8)(4.3,8)(4.6,7.5)

\bezier{200}(4,8)(6,8.4)(8.9,8.4)
\bezier{200}(4,8)(6,7.6)(8.9,7.6)

\bezier{200}(8.9,8.4)(8.7,8.2)(8.7,7.9)
\bezier{200}(8.9,8.4)(9.1,8.5)(9.1,8.1)
\bezier{200}(8.9,7.6)(8.7,7.5)(8.7,7.9)
\bezier{200}(8.9,7.6)(9.1,7.8)(9.1,8.1)

\bezier{200}(5,6.5)(5.3,6.5)(5.6,7)
\bezier{200}(5,6.5)(5.3,6.5)(5.6,6)

\bezier{200}(5,6.5)(6.7,6.9)(9.1,6.9)
\bezier{200}(5,6.5)(6.7,6.1)(9.1,6.1)

\bezier{200}(9.1,6.9)(8.9,6.7)(8.9,6.4)
\bezier{200}(9.1,6.9)(9.3,7)(9.3,6.6)
\bezier{200}(9.1,6.1)(8.9,6)(8.9,6.4)
\bezier{200}(9.1,6.1)(9.3,6.3)(9.3,6.6)

\bezier{300}(4,2.3)(6.5,2.3)(8.9,2.1)
\bezier{300}(5,1.7)(7,1.7)(9.1,1.9)

\put(9.9,1.9){\makebox(0,0)[bl]{$\bullet$}}
\put(10.3,1.7){\makebox(0,0)[bl]{$\eta\cdot \frac{z}{|z|}$}}
\end{picture}
\end{center}
\begin{proposition}[\cite{He4}, Section 8.1]
The tuple $(\cH,\nabla)$ constructed above using Lefschetz thimbles and
oscillating integrals is exactly the Fourier-Laplace-transformation of the
Brieskorn lattice $M_0$.
\end{proposition}

In the case of a (deformation of a) local singularity or a non-tame
deformation of a tame function, a good representative $F_{\underline{y}}:U\to
S$ exists, but no extension to $\infty$. So one cannot work directly with
oscillating integrals as above. But one can imitate them on a cohomological
level \cite{He4,DS1}. In any case,
if we are starting with an unfolding of a germ or of a tame function parameterized by $M$,
then a variation of the above construction yields a holomorphic vector bundle $H$ over $\dC\times M$,
with a connection with pole of type $1$ (also called
Poincar{\'e} rank one) along $\{0\}\times M$, i.e., the sheaf $\cH$ is stable under $z^2\nabla_z$ and
$z\nabla_X$ for any $X\in\cT_M$.

\section{Twistor structures}
\label{secTwistor}

Here we define and study a framework which encompasses the
objects described in the last section. The main objects are
twistor and TERP-structures. The former have been
introduced in \cite{Si5}, for the latter, see \cite{He4} and
\cite{HS1}.
\begin{definition}
\label{defTERP}
Let $M$ be a complex manifold. A variation of TERP-struc\-tures on $M$ of weight $w\in \dZ$ is
a tuple $\TERP$ consisting of:
\begin{enumerate}
\item
a holomorphic vector bundle
$H$ on $\dC\times M$, equipped with a meromorphic connection
$\nabla:\cH\rightarrow \cH\otimes z^{-1}\Omega^1_{\dC\times
M}(\log\{0\}\times M)$, satisfying $\nabla^2=0$.
The restriction $H':=H_{|\dC^*\times M}$ is then necessarily flat, corresponding
to a local system given by a monodromy automorphism $M\in\Aut(H^\infty)$,
where $H^\infty$ is the space of flat multivalued sections. For simplicity,
we will make the assumption that the eigenvalues of $M$ are
in $S^1$, which is virtually always the case in applications.
\item
a flat real subbundle $H'_\dR$ of maximal rank of
the restriction $H':=H_{\dC^*\times M}$. In particular,
$M$ is actually an element in $\Aut(H^\infty_\dR)$.
\item
a
$(-1)^w$-symmetric, non-degenerate and flat pairing
$$
P:\cH'\otimes j^*\cH'\rightarrow \cO_{\dC^*}
$$
where $j(z,x)=(-z,x)$, and which takes values in $i^w\dR$ on the real
subbundle $H'_\dR$. It extends to a pairing on $i_*\cH'$, where
$i:\dC^*\times M\hookrightarrow \dC\times M$, and has the
following two properties on the subsheaf $\cH\subset i_*\cH'$:
\begin{enumerate}
\item
$P(\cH,\cH)\subset z^w\cO_{\dC\times M}.$
\item
$P$ induces a non-degenerate symmetric pairing
$$
[z^{-w}P]:\cH/z\cH\otimes \cH/z\cH\rightarrow \cO_M.
$$
\end{enumerate}
\end{enumerate}
If no confusion is possible, we will denote a variation of TERP-structure
by its underlying holomorphic bundle. A variation $H$ over $M=\{pt\}$ is
called a single TERP-structure.

A TERP-structure is called regular singular if $\nabla$
has a regular singularity along $\{0\}\times M$.
\end{definition}
\begin{theorem}[\cite{He4,DS1}]
The Fourier-Laplace transformation of the Brieskorn lattice of an isolated hypersurface singularity
$f:(\dC^{n+1},0)\rightarrow(\dC,0)$ or a tame function $f:X\rightarrow \dA^1_\dC$
underlies a TERP-structure, where $P$ is induced by the intersection form on Lefschetz thimbles,
and $w=n+1$. It is regular singular in the first case and in general irregular
in the second case. Any unfolding $F$ of a local singularity or a tame function yields a variation of TERP-structures,
which is regular singular only if $f$ is local and the unfolding is
a $\mu$-constant deformation.
We will denote any such (variation of) TERP-structure(s) by $\mathit{TERP}(f)$
(respectively $\mathit{TERP}(F)$).
\end{theorem}
We will see that any TERP-structure gives rise
to a twistor structure. This connects singularity theory
to an priori completely different area, namely, the
theory of harmonic bundles. We first recall the basic definitions, our main
reference is  \cite{Si2,Si5,Mo2}.
\begin{definition}
A twistor is a holomorphic bundle on $\dP^1$. A family of twistors over $M$ is
a complex vector bundle $F$ on $\dP^1\times M$, equipped with a $\dP^1$-holomorphic
structure (i.e. the corresponding sheaf $\cF$ of sections is a locally free $\cO_{\dP^1}\cC^\infty_M$-module).
$F$ is called pure (of weight zero) if it is fibrewise trivial, i.e., if
$\cF=p^*p_*\cF$ (here $p:\dP^1\times M\rightarrow M$
is the projection and the functor
$\cO_{\dP^1}\cC^\infty_M\otimes_{p^{-1}\cC^\infty_M}p^{-1}(-)$ is denoted by $p^*$).
To define a variation and a polarization of twistor structures,
we need the map $\sigma:\dP^1\times M\to \dP^1\times M$, $\sigma(z,x)=(-\overline{z}^{-1},x)$
and the two sheaves of meromorphic functions
$\cO_{\dP^1}(1,0)=l_*\cO_{\dP^1\backslash \{0\}}\cap \widetilde{i}_*z^{-1}\cO_\dC$ and
$\cO_{\dP^1}(0,1)=l_*z\cO_{\dP^1\backslash \{0\}}\cap \widetilde{i}_*\cO_\dC$
where $l:\dP^1\backslash \{0\}\hookrightarrow \dP^1$ and $\widetilde{i}:\dC\hookrightarrow \dP^1$.
A variation of twistor structures is a family $F$ of twistor structures, together with an operator
$$
\mathbf{D}:\cF\longrightarrow\cF\otimes(\cO_{\dP^1}(1,0)\otimes\cA^{1,0}_M \oplus \cO_{\dP^1}(0,1)\otimes\cA^{0,1}_M)
$$
such that $\mathbf{D}$ is $\cO_{\dP^1}$-linear and
$\mathbf{D}_{|\{z\}\times M}$ is a flat connection for any $z\in\dC^*$.
Then $\mathbf{D}''_{|\dC\times M}$ defines a holomorphic structure on $\cF_{|\dC\times M}$
and $\mathbf{D}'_{|\dC\times M}$ gives a family of flat connections, with a
pole of order 1 along $\{0\}\times M$,
and  $\mathbf{D}'_{|(\dP^1-\{0\})\times M}$ defines an antiholomorphic structure on $\cF_{|(\dP^1-\{0\})\times M}$
and $\mathbf{D}''_{|(\dP^1-\{0\})\times M}$ gives a family of flat
connections, with antiholomorphic pole of order 1 along $\{\infty\}\times M$.
In the case of a variation $\cF$ of pure twistors (of weight zero)
a polarization of $\cF$ is a symmetric $\mathbf{D}$-flat non-degenerate pairing
$$
S:\cF\otimes\sigma^*\cF\longrightarrow \cO_{\dP^1}
$$
which is a morphism of twistors. It induces a Hermitian pairing $h$ on $E:=p_*\cF$. Then $F$ is called
polarized if $h$ is positive definite.
\end{definition}
Given a pure variation, the connection $\mathbf{D}$ induces operators
on $E$. If the variation is polarized these operators satisfy some natural compatibility conditions.
More precisely, the metric $h$ on $E$ is harmonic, meaning that it corresponds
to a pluriharmonic map from the universal cover of $M$ to the space of positive
definite Hermitian matrices. In terms more closely related to our situation,
we can define $(E,h)$ to be harmonic if $(D'+D''+\theta+\overline{\theta})^2=0$, where
$D''$ defines the holomorphic structure on $E$ which corresponds to
$F_{|\{0\}\times M}$, $\theta$ is defined by the class of $z\mathbf{D}'_{\partial z}\in{\cE}\!nd_{\cO_M}(
\cF/z\cF)$, $D'+D''$ is $h$-metric and $\overline{\theta}$ is
the $h$-adjoint of $\theta$ (see also lemma \ref{lemCV} below).

Vice versa, given $(E,h)$ with such operators $D''$ and $\theta$, one might reconstruct the whole
variation $(F:=p^*E,\mathbf{D},S)$ from them. The basic correspondence due to Simpson can be stated as follows.
\begin{theorem}[\cite{Si5}, Lemma 3.1]
The category of variations of pure polarized twistor structures on $M$ is equivalent to the category
of harmonic bundles.
\end{theorem}

We are going to construct a twistor for any given TERP-structure $H$. For this purpose,
define $\gamma:\dP^1\times M\rightarrow \dP^1\times M$ to be $\gamma(z,x)=(\overline{z}^{-1},x)$.
Consider the bundle $\overline{\gamma^*H}$ (here $\overline{\cdot}$ denotes
the conjugate complex structure in the fibres of a bundle). It is a holomorphic bundle on the
complex manifold $(\dP^1\backslash \{0\})\times \overline{M}$
(there is no need for a conjugation in the $\dP^1$-direction as this is already built-in
in the definition of the bundle $\overline{\gamma^*H}$). We define an identification
$\tau:H_{|\dC^*\times M}\stackrel{\cong}{\longrightarrow}\overline{\gamma^*H}_{|\dC^*\times M}$ by
$$
\hspace*{-1.3cm}
\setlength{\unitlength}{0.45cm}
\begin{picture}(12,6)
\bezier{300}(9,0)(6,0)(6,3)
\bezier{300}(9,0)(12,0)(12,3)
\bezier{300}(9,6)(6,6)(6,3)
\bezier{300}(9,6)(12,6)(12,3)
\bezier{300}(9.3,1)(11,3)(9.3,5)
\put(9,0.5){\makebox(0,0)[bl]{$\cdot\ 0$}}
\put(9,5.2){\makebox(0,0)[bl]{$\cdot\ \infty$}}
\put(10.3,1.8){\makebox(0,0)[bl]{$z$}}
\put(9.9,2){\circle*{0.2}}
\put(10.3,3.8){\makebox(0,0)[bl]{$\gamma(z)$}}
\put(9.9,4){\circle*{0.2}}
\end{picture}
\hspace*{1.2cm}
\begin{array}[b]{rcl}
\tau: H_{z,x} & \longrightarrow & \overline{H}_{\gamma(z),x}\\
s & \longmapsto & \nabla\textup{-parallel transport  of }\overline{z^{-w}s}.
\end{array}
$$
\begin{proposition}[\cite{He4}, Chapter 2]\label{propPropOfTau}
$\tau$ is a linear involution, it identifies
$H_{|\dC^*\times M}$ with $\overline{\gamma^*H}_{|\dC^*\times M}$ and defines a
(variation of) twistor(s) $\widehat{H}:=H\cup_\tau\overline{\gamma^*H}$. Moreover,
$\tau$ acts as an anti-linear involution on $E=p_*\widehat{H}$.
We obtain a pairing $S$ on the twistor $\widehat{H}$ by $S:=z^{-w}P(-,\tau-)$.
Each bundle $\widehat{H}_{|\dP^1\times\{x\}}$, $x\in M$, has degree zero.
\end{proposition}
By analogy with the notion of twistors, we will make the following definition.
\begin{definition}
$H$ is called a pure TERP-structure iff $\widehat{H}$ is pure, i.e. fibrewise trivial. In that case,
it is called pure polarized iff $\widehat{H}$ is so, i.e., iff $h:=S_{|p_*\widehat{H}}$ is positive definite.
\end{definition}
The following lemma is a straightforward calculation; it gives a reformulation of the whole
structure in terms of the bundle $E$, if the structure is pure.
In particular, it allows one to define a very interesting object associated to a TERP-structure.
\begin{lemma}[\cite{He4}, Theorem 2.19]
\label{lemCV}
For any variation of pure TERP-structu\-res, the connection operator $\nabla$ takes the following
form on fibrewise global sections
$$
\nabla = D'+D''+z^{-1}\theta+z\overline{\theta}+\frac{dz}{z}\left(\frac{1}{z}\cU-\cQ+\frac{w}{2}-z\tau\cU\tau\right)
$$
\end{lemma}
The symbols used in this expressions are mappings
$$
\begin{array}{c}
D',\theta:\cC^\infty(E)\longrightarrow \cC^\infty(E) \otimes\cA^{1,0}_M \\
D'',\overline{\theta}:\cC^\infty(E) \longrightarrow \cC^\infty(E) \otimes\cA^{0,1}_M \\
\cU,\cQ\in{\cE}\!nd_{\cC^\infty_M}(E)
\end{array}
$$
where $D'$ and $D''$ satisfy the Leibniz rule whereas $\theta$ and $\overline{\theta}$ are
linear over $\cC^\infty_M$.
These objects satisfy the following relations, summarized under the name CV$\oplus$-structure in \cite{He4}.
\begin{eqnarray}
\label{eqHarmBund1}h(\theta-,-)-h(-,\overline{\theta})=0,&&(D'+D'')(h)=0\\
\label{eqHarmBund2}(D''+\theta)^2=0,&&(D'+\overline{\theta})^2=0 \\
\label{eqTT*1}D'(\theta)=0,&&D''(\overline{\theta})=0 \\
\label{eqTT*2}D'D''+D''D'&=&-(\theta\overline{\theta}+\overline{\theta}\theta)\\
\label{eqInt1}[\theta,\cU]=0,&&D'(\cU)-[\theta,\cQ]+\theta=0\\
\label{eqInt2}D''(\cU)=0,&&D'(\cQ)+[\theta,\tau\cU\tau]=0\\
\nonumber \tau\theta\tau=\overline{\theta}&&(D'+D'')(\tau)=0\\
\nonumber h(\cU-,-)=h(-,\tau\cU\tau-),&&h(\cQ-,-)=h(-,\cQ-)\\
\nonumber \cQ&=&-\tau\cQ\tau
\end{eqnarray}
As already remarked, equations \eqref{eqHarmBund1} to \eqref{eqTT*2} say that the metric $h$ on $E$ is harmonic.
The two identities \eqref{eqTT*1} and \eqref{eqTT*2} were called $tt^*$-equations
in \cite{CV1}. Finally, variation of twistor structures corresponding to
harmonic bundles with operators $\cU,\cQ, \tau\cU\tau$
are studied under the name ``integrable'' in \cite[Chapter 7]{Sa6},
and the identities \eqref{eqInt1} and \eqref{eqInt2} are called ``integrability equations'' in loc.cit.

A particularly interesting piece of this structure is the endomorphism $\cQ$.
\begin{lemma}\cite[Lemma 2.18]{He4}
If $H$ is pure polarized, then
$\cQ$ is a Hermitian endomorphism of the bundle $(E,h)$ and its real-analytically varying (real)
eigenvalues are distributed symmetrically around zero.
\end{lemma}
$\cQ$ was already considered in \cite{CFIV} under the name ``new supersymmetric index''. We will describe
some more results concerning its eigenvalues in section \ref{secLimit}.

\section{Nilpotent orbits}
\label{secNilpotent}

In this section we will consider a particular one-parameter variation of TERP-structures
which gives rise to a very interesting correspondence with polarized mixed Hodge structures
(PMHS). This relies on the construction of a filtration on the space $H^\infty$. We will give two
versions of this construction adapted to different situations. It can be easily checked that
these families are in fact variations of TERP-structures on $\dC^*$.

Let $\TERP$ be a given TERP-structure. We define
two one-parameter variations $K$, resp. $K'$, by taking the pull-backs
$K:=\pi^*H$, resp. $K':=(\pi')^*H$, where
$\pi,\pi':\dC\times\dC^*\rightarrow \dC$ are defined
by $\pi(z,r)=zr$ and $\pi'(z,r)=zr^{-1}$.
\begin{definition}
$K$ (resp. $K'$) is called a nilpotent orbit (resp. Sabbah orbit) of TERP-structures iff
its restriction $K$, resp. $K'$,
to $\Delta^*_r:=\left\{z\in\dC^*\,|\,|z|<r\right\}$ is pure polarized for
$|r|\ll 1$.
\end{definition}
As explained in section \ref{secBrieskorn}, the main sources of TERP-structures are Brieskorn lattices
defined by holomorphic functions or germs. In the local case, a filtration on $H^\infty$ has been defined in
\cite{Va1} and was later modified in \cite{SchSt}. In our language, these definitions are valid for regular singular TERP-structures.
In the general case, we give a definition due to Sabbah \cite{Sa2}.

We use the Deligne extensions defined earlier. They induce a filtration, called a
$V$-filtration,
on the Brieskorn lattice so that one might consider the corresponding graded object.
These graded parts determine the filtrations we are looking for.
\begin{definition}
Let $\TERP$ be a given TERP-structure of weight $w$.
\begin{enumerate}
\item
Suppose that $(H,\nabla)$ is regular singular at zero (i.e. $\cH\subset V^{>-\infty}$). Then put
for $\alpha\in (0,1]$
\begin{eqnarray}
\label{eqTERPHodge}
F^pH^\infty_{e^{-2\pi i \alpha}}:=z^{p+1-w-\alpha+\frac{N}{2\pi i}}Gr_V^{\alpha+w-1-p}\cH
\end{eqnarray}
\item
Let $(H,\nabla)$ be arbitrary
and put $G_0:=H^0(\dP^1, l_* V_{<\infty}\cap\widetilde{i}_*\cH)$
(where, as before, $\widetilde{i}:\dC\hookrightarrow \dP^1$ and $l:\dP^1\backslash\{0\}\hookrightarrow\dP^1$).
Then $G_0$ is a free $\dC[z]$-module
and the Deligne extensions at infinity induce a filtration on it so that we can define for $\alpha\in (0,1]$
\begin{eqnarray}
\label{eqTERPHodgeSabbah}
F^p_{Sab}H^\infty_{e^{-2\pi i \alpha}}
:=z^{p+1-w-\alpha+\frac{N}{2\pi i}}Gr_{\alpha+w-1-p}^V G_0
\end{eqnarray}
\end{enumerate}
\end{definition}
We will actually not work with $F^\bullet$ and
$F^\bullet_{Sab}$ directly, but with a twisted version defined
by $\widetilde{F}^\bullet:=G^{-1}F^\bullet$ and similarly for $F^\bullet_{Sab}$, where
$G:=\sum_{\alpha \in(0,1]} G^{(\alpha)} \in
\Aut\left(H^\infty=\oplus_{\alpha}
H^\infty_{e^{-2\pi i\alpha}}\right) $ is defined as follows (see \cite[(7.47)]{He4}):
$$
G^{(\alpha)} := \sum_{k\geq 0}\frac{1}{k!}\Gamma^{(k)}(\alpha)
\left( \frac{-N}{2\pi i}\right)^k
=: \Gamma \left(\alpha\cdot {\rm id} - \frac{N}{2\pi i}\right) .
$$
Here $\Gamma^{(k)}$ is the $k$-th derivative of the gamma function.
In particular, $G$
depends only on $H'$ and induces the identity on $\Gr^W_\bullet$, $W_\bullet$ being the weight filtration of $N$.

Note that for $H=\mathit{TERP}(f)$, where $f$ is a local singularity,
Steenbrink's Hodge filtration from \cite{SchSt} is exactly our $\widetilde{F}^\bullet$.
It was defined in loc.cit. by a formula similar to \eqref{eqTERPHodge}, but using
the Brieskorn lattice $\cH^{(0)}$ and its V-filtration.
On the other hand, the filtration $F^\bullet$ is defined using the Fourier-Laplace
transform $\cH$ of $\cH^{(0)}$, so that $G$ can be seen as the ``topological part''
of the Fourier-Laplace transformation.

The next ingredient we need is a polarizing form $S$ induced on $H^\infty$ by the pairing $P$.
It can be defined as follows: $P$ induces a pairing (denoted by the same symbol) on the
local system $(H')^\nabla$, then given $A,B\in H^\infty$, we put
$S(A,B):= (-1)(2\pi i)^w P(A, t(B))$ where
$$
t(B)=
\left\{
\begin{array}{c}
(M-{\rm Id})^{-1}(B) \;\;\;\;\;\; \forall B\in H^\infty_{\neq 1}\\ \\
-(\sum\limits_{k\geq1}\frac{1}{k!}N^{k-1})^{-1}(B)\;\;\;\;\;\; \forall B\in H^\infty_1.
\end{array}
\right.
$$
$S$ is nondegenerate, monodromy invariant, $(-1)^w$-symmetric on $H^\infty_1$, $(-1)^{w-1}$-symmetric on $H^\infty_{\neq 1}$,
and it takes real values on $H^\infty_\dR$ \cite[Lemma 7.6]{He4}.
A basic result concerning $\widetilde{F}^\bullet$ and $S$ is the following.
\begin{theorem}[\cite{Va1,SchSt,He3}]\label{theoVarSchSt}
For a local singularity $f$, the tuple $(H^\infty,H_\dR^\infty, -N, S, \widetilde{F}^\bullet)$ is a
(sum of) polarized mixed Hodge structure(s) (of weight $w$ on $H^\infty_{1}$ and of weight
$w-1$ on $H^\infty_{\neq 1}$).
\end{theorem}

We come back to our special situation of a family $K:=\pi^*H$. Suppose first that
it is regular singular (it suffices actually to check this at any value $r$, e.g., for
$H=K_{|r=1}$).
\begin{theorem}[\cite{He4,HS1}]
\label{theoCorrMHSNilpot}
The family $K$ is a nilpotent orbit of TERP-structures, i.e., it is pure polarized
on a disc with sufficiently small radius iff
$(H^\infty, H^\infty_\dR, -N, S, \widetilde{F}^\bullet)$
defines a PMHS of weight $w-1$ on $H^\infty_{\neq 1}$ and
a PMHS of weight $w$ on $H^\infty_1$.
\end{theorem}
Let us give some comments on the proof: The direction "$\Leftarrow$" is proved in
\cite{He4}, and relies on the corresponding correspondence
between PMHS and nilpotent orbits of Hodge structures (\cite{Sch,CaKS}).
The opposite direction "$\Rightarrow$" is done in \cite{HS1}, and uses a different strategy:
It is built on a central result in \cite{Mo2} which states that a tame variation
of pure polarized twistors on, say $\Delta^*$, degenerates to what is called a polarized
mixed twistor structure (see section \ref{secLimit}). In our situation, this
limit object corresponds exactly to the (sum of) PMHS('s) on $H^\infty$, which
shows the desired result.

For arbitrary TERP-structures, we have a corresponding statement
for Sabbah orbits, which is proven in exactly the same way.
\begin{theorem}[\cite{HS1}]
\label{theoCorrMHSSabbah}
Let $\TERP$ be arbitrary. Then the family $K'$ is a Sabbah orbit
on $\Delta^*_r$ for sufficiently small $r$ iff
$(H^\infty, H^\infty_\dR, N, S, \widetilde{F}_{Sab}^\bullet)$
defines a PMHS of weight $w-1$ on $H^\infty_{\neq 1}$ and
a PMHS of weight $w$ on $H^\infty_1$.
\end{theorem}
\begin{remark}
In \cite{HS1} a conjecture is stated which generalizes the correspondence in
Theorem \ref{theoCorrMHSNilpot} to the case of general (i.e. irregular) TERP-structures.
The notion of a nilpotent orbit is the same, but the other side, the PMHS, has
to be replaced by more subtle conditions: First, it is assumed that the formal
decomposition of the germ $(\cH,\nabla)_0$ into model bundles exists without
ramification. Second, Stokes structure and real structure shall be compatible.
Third, the regular singular parts of the model bundles shall induce PMHS just
as above. The implication from this generalization of a PMHS to a nilpotent
orbit is proved in \cite{HS1}, the other implication is still open.
Notice that in the case of a TERP-structure of a deformation of a germ or a
tame function the three conditions for the generalization of a PMHS are all
satisfied: a ramification is not necessary, the regular singular parts of the
model bundles are the (Fourier-Laplace transforms of the) Brieskorn lattices
of the local singularities, Theorem \ref{theoVarSchSt} gives their PMHS, and the
real and Stokes structure are compatible because both are defined by Lefschetz thimbles.
\end{remark}

We will give some comments on applications of these results. First note the following simple
lemma, which follows from the definition of the Gau{\ss}-Manin connection on the Brieskorn lattice
given in section \ref{secBrieskorn}. It shows how nilpotent orbits arise naturally
in singularity theory.
\begin{lemma}[\cite{HS1}, lemma 11.3]
Let $f$ be a germ of a holomorphic function with isolated singularities or a tame
function on an affine manifold. Then for any $r\in\dC$
$$
\mathit{TERP}(r\cdot f) = \pi'(\mathit{TERP}(f))_{|r}
$$
Moreover, in the local case, the Hodge filtration $\widetilde F^\bullet(f)$ associated
to $\mathit{TERP}(f)$ satisfies.
$$
\widetilde F^\bullet(r\cdot f) = r^{\frac{N}{2\pi i}} \widetilde F^\bullet(f).
$$
\end{lemma}
An easy consequence of the second part is the following result, which
uses only classical facts about nilpotent orbits of Hodge structures.
\begin{corollary}
For any function germ $f$ and any sufficiently large $r\in\dR$, we have that
$\widetilde{F}^\bullet(r\cdot f)$ gives a pure polarized Hodge structure on $H^\infty$
(of weight $w$ on $H^\infty_1$ and of weight $w-1$ on $H^\infty_{\neq 1}$).
\end{corollary}

The following application uses the generalization to the irregular case discussed above
and the geometry on the unfolding space of the germ $f$ as discussed in detail in \cite{He3}.
\begin{corollary}
Let $F:(\dC^{n+1}\times M,0)\rightarrow (\dC,0)$ be
a semi-universal unfolding of an isolated hypersurface singularity
$f=F_{|\dC^{n+1}\times\{0\}}:(\dC^{n+1},0)\rightarrow(\dC,0)$.
Then for any $\underline{y}_0\in M$, $\mathit{TERP}(r\cdot
F_{\underline{y}_0})$ is pure polarized for $r\gg 0$.
Note that the families $\{r\cdot F_{\underline{y}}\ |\ r\in \dC^*\}$ correspond to
the orbits of the Euler field $E$ on $M$, so this result says that going sufficiently
far in $M$ along ${\rm Re}(E)$, one always arrives at pure polarized TERP-structures.
\end{corollary}
In the case of tame functions, we can use the other direction (of Theorem \ref{theoCorrMHSSabbah})
to strengthen a result of Sabbah.
\begin{corollary}[\cite{HS1}, Corollary 11.4]
Let $f:X\rightarrow \dA^1_\dC$ be a given tame function, then the tuple
$(H^\infty,H_\dR^\infty, N, S, \widetilde{F}^\bullet_{Sab})$ associated
to $\mathit{TERP}(f)$ is a polarized mixed Hodge structure.
\end{corollary}
In \cite{Sa2} it was shown that $(H^\infty,H_\dR^\infty, N, F^\bullet_{Sab})$ is a mixed Hodge structure.
The corollary follows from Theorem \ref{theoCorrMHSSabbah} using another, more recent result,
which we believe to be of fundamental importance
in the theory of twistor structures associated to tame functions.
It already appears in a completely different language in \cite{CV1} and \cite{CV2}.
\begin{theorem}[\cite{Sa8}, see also \cite{Sa9}]
\label{theoTamePurePolarized}
Let $f:X\rightarrow \dA^1_\dC$  be a tame function. Then $\mathit{TERP}(f)$ is pure
polarized.
\end{theorem}

\section{Classifying spaces and curvature}
\label{secClassifying}

In order to study period mappings for $\mu$-constant
deformations of hypersurface singularities, appropriate
classifying spaces for Brieskorn lattices were discussed in
\cite{He2}. We first give a short review of this construction,
then we show how to endow a part (called pure polarized) of
these classifying spaces with a Hermitian metric.

An important result in the theory of variation of Hodge
structures is that the curvature of the Hermitian
metric on the classifying spaces of Hodge structures is negative along horizontal directions. It turns
out that a similar result holds for variation of
TERP-structures; this will also be explained in the current
section.

The first ingredient for constructing the classifying spaces is
a classical cohomological invariant, the spectrum, attached to an isolated hypersurface
singularity, which can in fact be defined for any regular singular TERP-structure.
Note that our definition gives real numbers shifted by one
compared to the original definition of Varchenko and Steenbrink.
These shifted numbers were used by M.~Saito under the name exponents.
\begin{definition}
Let $\TERP$ be a regular singular TERP-structure of weight $w$.
\begin{enumerate}
\item
The spectrum of $(H,\nabla)$ at zero is defined as
$\Sp(H,\nabla) = \sum_{\alpha \in\dQ}d(\alpha)\cdot\alpha\in\dZ[\dQ]$
where
$$
d(\alpha):=\dim_\dC\left(\frac{Gr^\alpha_V \cH}{Gr^\alpha_V z\cH}\right)
=\dim_\dC\Gr_F^{\lfloor w-\alpha\rfloor} H^\infty_{e^{-2\pi i \alpha}}.
$$
We also write $\Sp(H,\nabla)$ as a tuple $\alpha_1, \ldots, \alpha_\mu$ of $\mu$ numbers (with $\mu=\rk(H)$),
ordered by $\alpha_1\leq ...\leq \alpha_\mu$. We have that
$\alpha_i=w-\alpha_{\mu+1-i}$ and that $\alpha$ is a spectral number only if $e^{-2\pi i \alpha}$ is an eigenvalue
of the monodromy $M$ of $H'$ (in particular, all $\alpha_i$ are real by assumption).
\item
The spectral pairs are a refinement of the spectrum taking into account the
weight filtration $W_\bullet(N)$ (Here the restriction of $W_\bullet(N)$ to $H^\infty_1$ is centered around
$w$, and the restriction to $H^\infty_{\neq 1}$ is centered around $w-1$).
They are given by
$\Spp(H,\nabla):=\sum_{\alpha \in\dQ}\widetilde{d}(\alpha,k)\cdot(\alpha,k)\in\dZ[\dQ\times\dZ]$, where
$$
\widetilde{d}(\alpha,k):=
\left\{
\begin{array}{ll}
\dim_\dC\Gr_F^{\lfloor w-\alpha\rfloor} \Gr^W_k H^\infty_{e^{-2\pi i \alpha}} & \forall\alpha\notin\dZ\\
\dim_\dC\Gr_F^{w-\alpha} \Gr^W_{k+1} H^\infty_1 & \forall\alpha\in\dZ
\end{array}
\right.
$$
\end{enumerate}
\end{definition}
For the remainder of this section, we fix $H^\infty,H^\infty_\dR, S, M, w, \Spp$
which we call initial data. Note that $(H^\infty,H^\infty_\dR, S, M)$ are equivalent
to the datum of a flat bundle $H'$ over $\dC^*$ with a flat real subbundle
and a pairing $P$ as in Definition \ref{defTERP}, where $S$ is induced from $P$ as in section \ref{secNilpotent}.
For notational convenience,
let $n:=\lfloor\alpha_\mu-\alpha_1\rfloor$.
We also fix a reference filtration
$F_0^\bullet$ such that $(H^\infty,H_\dR^\infty,N, S, F^\bullet_0)$ is a PMHS.
Denote by $P_l\subset \Gr_l^W$ the primitive subspace of the weight filtrations $W_\bullet$ of $N$.
Then put
\begin{eqnarray*}
\DcPMHS&:=&\left\{F^\bullet H^\infty\,|\,
\dim F^pP_l=\dim F^p_0P_l;\ S(F^p,F^{w+1-p})=0; \right.   \\
&& \left. N(F^p)\subset F^{p-1}, \mbox{ and all powers of }N\mbox{ are strict with respect to }F^\bullet
\right\}
\end{eqnarray*}
There is a projection map $\check{\beta}:\DcPMHS\rightarrow \DcPHS$,
sending $F^\bullet$ to $(F^\bullet P_l)_{l\in\dZ}$, here
$\DcPHS$ is a product of classifying spaces of Hodge-like filtrations
on the primitive subspaces $P_l$. $\DcPHS$ is a complex homogeneous
space, and it contains the open submanifold $\DPHS$ which is a product of classifying
spaces of polarized Hodge structures, and which has the structure of a real
homogeneous space. $\check{\beta}$ is a locally trivial fibration
with affine spaces as fibres. Define $\DPMHS$ to be the restriction of this fibration to
$\DPHS$. Then also $\DcPMHS$ (resp. $\DPMHS$) is a complex (resp. real) homogeneous space.

One can construct period maps for singularities by associating
to a singularity $f$ its Hodge filtration $\widetilde
F^\bullet$. However, it is readily seen that in
general the information encoded in the Hodge structure is too
weak for, say, Torelli theorems. The basic philosophy is that
the whole Brieskorn lattice should be sufficiently rich to
determine the singularity. This calls for a classifying space
for Brieskorn lattices, which is constructed as follows. Put
\begin{eqnarray*}
\cbl & := &
\big\{
\cH\subset V^{\alpha_1}\,|\, H\rightarrow \dC\mbox{ holomorphic vector bundle},H_{|\dC^*}=H', \\
&&(z^2\nabla_z)(\cH)\subset \cH, P(\cH_0,\cH_0)\subset z^w\dC\{z\}\mbox{ non-degenerate},\widetilde F^\bullet_H
\in\DcPMHS\big\}
\end{eqnarray*}
Remember that $F^\bullet_H$ is the filtration defined by $H$ using formula \eqref{eqTERPHodge} on $H^\infty$,
and $\widetilde F^\bullet_H=G^{-1}F^\bullet_H$.
As before, we obtain a projection map
$\check{\alpha}:\cbl\rightarrow\DcPMHS$ by $H\mapsto \widetilde{F}^\bullet_H$.
One of the main results of \cite{He2} is that this map is still a locally trivial
fibration where the fibres are affine spaces. Again we put $D_{\mathit{BL}}:=\check{\alpha}^{-1}(\DPMHS)$.
The situation can be visualized in the following diagram.
$$
\begin{array}{lcl}
D_{\mathit{BL}} & \hookrightarrow & \cbl \\
\downarrow{\scriptstyle\alpha}  & & \downarrow{\scriptstyle\check{\alpha}}   \\
\DPMHS & \hookrightarrow & \DcPMHS \\
\downarrow{\scriptstyle\beta}  & & \downarrow{\scriptstyle\check{\beta}}   \\
\DPHS & \stackrel{\gamma}{\hookrightarrow} & \DcPHS
\end{array}
$$
Note that neither $\cbl$ nor
$D_{BL}$ are homogeneous. However, there is a good $\dC^*$-action on the fibers of
$\cbl$; the corresponding zero section
$\DcPMHS\hookrightarrow \cbl$
consists of regular singular TERP-structures in $\cbl$
which are generated by elementary sections, see \cite[Theorem 5.6]{He2}.

Now suppose that we have in addition to the initial data a lattice
$H^\infty_\dZ \subset H^\infty_\dR$ such that $M\in\Aut(H^\infty_\dZ)$.
Then the discrete group $G_\dZ:=\Aut(H^\infty_\dZ,S,M)$ acts properly discontinuously on $\DPMHS$
and $D_{BL}$ with normal complex spaces as quotients \cite[Prop. 2.6 and Cor. 5.7]{He2}.
For any $\mu$-constant family of isolated hypersurface singularities
parameterized by a manifold $M$, we obtain period maps
$\phi_{\mathit{PMHS}}:M\rightarrow \DPMHS/G_\dZ$ and
$\phi_{BL}:M\rightarrow D_{BL}/G_\dZ$ which are holomorphic and locally
liftable (here the fact, due to Varchenko \cite{Va}, that for such a family
the spectral pairs are constant is used).

{}From the very construction of the classifying spaces, we have the following fact.
\begin{lemma}[\cite{HS2}, Lemma 3.1]
The space $\cbl$ comes equipped with a universal bundle $H\rightarrow \dC\times\cbl$ of Brieskorn
lattices. $H$ underlies a family of TERP-structures on $\cbl$, but not a variation; more precisely,
we have a connection operator
$$
\nabla:\cH\longrightarrow \cH\otimes\left(z^{-2}\Omega^1_{\dC\times\cbl/\cbl}\oplus
z^{-n}\Omega^1_{\dC\times\cbl/\dC}\right)
$$
\end{lemma}
We put
$$
\cblp:=\{x\in \cbl\,|\, \widehat{H}_{|x}\mbox{ is pure polarized }\}
$$
We obtain a Hermitian metric on $p_*\widehat{H}_{|\cblp}$ by definition. A nice feature
of the space $\cbl^{pp}$ is that this metric can be transferred canonically
to its tangent bundle.
\begin{proposition}
The tangent sheaf $\cT_{\cbl}$ is naturally a subbundle of
$$
\widetilde{\cH}:={\cH}\!om_{\cO_{\dC\times\cbl}}\left(\cH,\frac{z^{-n}\cH}{\cH}\right) .
$$
Here $\cH$ and $z^{-n}\cH$ are seen as submodules of
$\cH\otimes\cO_{\dC\times\cbl}(*\{0\}\times\cbl)$.
There is a coherent subsheaf $\cT^{hor}_{\cbl}$ of $\cT_{\cbl}$
such that for any variation of TERP-structures on a simply connected space $M$,
we have that $d\phi_{BL}(\cT_M)\subset\phi_{BL}^*(\cT^{hor}_\cbl)$
(e.g. this holds for the local lifts of the period map from above).
On the subspace $\cblp\subset \cbl$, the tangent bundle
$\cT_{\cblp}$ is naturally equipped with a positive definite
Hermitian metric $h$.
\end{proposition}
\begin{proof}
First note that even though $\cH$ is an $\cO_{\dC\times\cbl}$-module,
the quotient $z^{-n}\cH/\cH$ is $z$-torsion and so is $\widetilde{\cH}$,
which might thus be considered as an $\cO_{\cbl}$-module. Then
the inclusion $\cT_{\cbl}\subset \widetilde{\cH} $
is given by the Kodaira-Spencer map
$$
\begin{array}{rcl}
\mathit{KS}:\cT_{\cbl} & \longrightarrow &\widetilde{\cH} \\
X &\longmapsto & \left[s\mapsto \nabla_X(s)\right].
\end{array}
$$
Injectivity follows from the universality of the classifying space $\cbl$.
The subsheaf $\cT^{hor}_{\cbl}$ is defined by
$$
\cT^{hor}_{\cbl}:=\cT_{\cbl}\cap
{\cH}\!om_{\cO_{\dC\times\cbl}}\left(\cH,\frac{z^{-1}\cH}{\cH}\right).
$$
from which its coherence is obvious. Computation of examples
shows that it is not necessarily locally free, except if $n=1$,
then $\cT^{hor}_{\cbl}=\cT_{\cbl}$. Consider now
the tangent bundle on the open subspace $\cblp$. By definition,
$H_{|\dC\times\cblp}$ is pure polarized, and $p_*\widehat{\cH}_{|\cblp}$ is a
locally free $\cC_{\cblp}^\infty$-module of rank $\mu$.
Let us denote by $\cR$ the sheaf of rings $\cO_\dC\cC^\infty_{\cblp}$ for a moment.
We write $\cH^{sp}$ for the sheaf $p_*\widehat{\cH}_{|\cblp}$.
It gives a splitting of the projection
$\cR\otimes\cH_{|\dC\times\cblp}\twoheadrightarrow
\cR\otimes(\cH/z\cH)_{|\dC\times\cblp}$, i.e.
$k^{-1}(\cR\otimes\cH_{|\dC\times\cblp}) =
\cH^{sp}\oplus k^{-1}(\cR\otimes(z\cH)_{|\dC\times\cblp})$, where
$k:\cblp\hookrightarrow \dC\times\cblp$, $t\mapsto(0,t)$.
This implies that
\begin{equation}
\label{eqMetricInduced}
\begin{array}{c}
\cC^\infty_{\cblp}\otimes\widetilde{\cH}_{|\cblp}\cong
{\cH}\!om_\cR\left(\cR\otimes \cH_{|\dC\times\cblp},\cR\otimes (\frac{z^{-n}\cH}{\cH}_{|\dC\times\cblp})\right)\cong\\ \\
\bigoplus_{k=1}^n{\cH}\!om_{\cC^\infty_{\cblp}}\left(\cH^{sp},z^{-k}\cH^{sp}\right)
\cong
\left[{\cE}\!nd_{\cC^\infty_{\cblp}}(\cH^{sp})\right]^n.
\end{array}
\end{equation}
The Hermitian metric on $p_*\widehat{\cH}_{|\cblp}$ induces
a metric on ${\cE}\!nd_{\cC^\infty_{\cblp}}(p_*\widehat{\cH}_{|\cblp})$
(by putting $h(\phi,\psi):=\Tr(\phi\cdot\psi^*)$,
where $\psi^*$ is the Hermitian adjoint of $\psi$) and on its powers and this
metric induces by restriction a positive definite metric on $\cT_{\cblp}$.
\end{proof}
The main result of \cite{HS2} is the following.
\begin{theorem}[\cite{HS2}, Theorem 4.1]
Consider the holomorphic sectional curvature
$$
\begin{array}{rcl}
\kappa: T_{\cblp}\backslash \{\mbox{zero-section}\} & \longrightarrow & \dR \\ \\
v &\longmapsto & h(R(v,\overline{v})v,v)/h^2(v,v)
\end{array}
$$
of the metric $h$. Here $R\in{\cE}\!nd_{\cC^\infty_{\cblp}}(\cT^\dC_{\cblp})\otimes\cA^{1,1}_{\cblp}$ is the curvature tensor
of the Chern connection of $h$.
Then the restriction of $\kappa$
to $T^{hor}_{\cblp}\backslash\{\mbox{zero section}\}$ (i.e., the complement of
the zero section of the linear space
associated to the coherent sheaf $\cT^{hor}_{\cblp}$) is bounded from above by a negative number.
\end{theorem}
The following is a direct consequence and proved as in the Hodge case using
Ahlfors' lemma.
\begin{corollary}[\cite{HS2}, Proposition 4.3]
\label{corAhlfors}
Any horizontal mapping $\phi:M\rightarrow \cblp$ from a complex manifold
$M$ is distance decreasing with respect to the Kobayashi pseudodistance on $M$
and the distance induced by $h$ on $\cblp$.
\end{corollary}
As an application, we obtain the following rigidity result.
\begin{corollary}[\cite{HS2}, Corollary 4.5]
Any variation of regular singular pure polarized TERP-struc\-tures with constant spectral pairs
on $\dC^m$ is trivial, i.e., flat in the parameter direction.
\end{corollary}
This result follows directly from the last corollary as
the Kobayashi pseudodistance of $\dC^m$ is zero, so that the corresponding
period map must be constant (see, e.g., \cite{Ko2}).

\section{Compactifications}
\label{secCompact}

A drawback of the construction of the last section is that the
metric $h$ on $\cblp$ is not complete. The reason is that
there exist families (even variations) of regular singular pure polarized TERP-structures, where the
spectral numbers jump for special parameters. Period mappings of such families lead naturally to
a partial compactification of $\cblp$.
The following example illustrates this phenomenon of jumping spectral numbers.
\begin{example}[\cite{HS3}, Section 9.2]
\label{exP114}
We will describe a variation of TERP-structures over
$\dC$ which extends to a variation over $\dP^1$ with jumping spectral
numbers at $\infty$.
Consider a three-dimensional real vector space $H^\infty_\dR$,
its complexification $H^\infty:=H^\infty_\dR\otimes \dC$, choose a basis
$H^\infty=\oplus_{i=1}^3\dC A_i$ such that $\overline{A}_1=A_3$, $A_2\in H^\infty_\dR$.
Choose a real number $\alpha_1\in(-3/2,-1)$, put $\alpha_2:=0$, $\alpha_3:=-\alpha_1$
and let $M\in\Aut(H^\infty_\dC)$ be given
by $M(\underline{A})=\underline{A}\cdot\diag(\lambda_1,\lambda_2,\lambda_3)$
where $\underline{A}:=(A_1,A_2,A_3)$ and $\lambda_i:=e^{-2\pi i \alpha_i}$ (then $M$ is actually an element in $\Aut(H^\infty_\dR)$).
Let, as before, $(H',H_\dR',\nabla)$ be the flat holomorphic bundle on $\dC^*\times\dC$ with real flat subbundle corresponding
to $(H^\infty,H^\infty_\dR,M)$, and put $s_i:=z^{\alpha_i}A_i\in\cH'$. Denote by $r$ a coordinate on $\dC$,
and define $\cH:=\oplus_{i=1}^3 \cO_{\dC^2} v_i$, where
$$
\begin{array}{rcl}
v_1 & := & s_1 + r z^{-1} s_2 + \frac{r^2}{2} z^{-2} s_3\\
v_2 & := & s_2 + r z^{-1} s_3\\
v_3 & := & s_3
\end{array}
$$
Moreover, define the pairing $P$ with $P:\cH'\otimes j^*\cH'\rightarrow \cO_{\dC^*\times \dC}$
by $P(\underline{s}^{tr}, \underline{s}):=$ $(\delta_{i+j,4})_{i,j\in\{1,\ldots,3\}}$.

It can be checked by direct calculations that
$\TERP$ is a variation of regular singular TERP-structures on $\dC$.
Moreover, the Hodge filtration induced on $H^\infty$ is constant in $r$ and
gives a sum of pure polarized Hodge structures of weights $0$ and $-1$ on $H^\infty_1$ and $H^\infty_{\neq 1}$; namely, we
have that
$$
\{0\}=F^2 \subsetneq
F^1:=\dC A_1 \subsetneq
F^0 := \dC A_1 \oplus \dC A_2 =
F^{-1} \subsetneq
F^{-2}:= H^\infty
$$
The polarizing form $S$ is defined by $P$ via \cite[(7.51), (7.52)]{He4}, namely, we
have:
$$
S(\underline{A}^{tr},\underline{A}):=
\begin{pmatrix}
0 & 0 & \gamma \\
0 & 1 & 0 \\
-\gamma & 0 &0
\end{pmatrix},
$$
where $\gamma:=\frac{-1}{2\pi i}\Gamma(\alpha_1+2)\Gamma(\alpha_3-1)$. In particular,
we have for $p=1$
$$
i^{p-(-1-p)}S(A_1,A_3)= (-1)iS(A_1,A_3)= \frac{\Gamma(\alpha_1+2)\Gamma(\alpha_3-1)}{2\pi} >0
$$
and for $p=0$
$$
i^{p-(-p)}S(A_2,A_2)=S(A_2,A_2)>0
$$
so that $F^\bullet$ indeed induces a pure polarized Hodge structure of weight $-1$ on $H^\infty_{\neq 1}=\dC A_1\oplus\dC A_2$
and a pure polarized Hodge structure of weight $0$ on $H^\infty_1=\dC A_2$.

The situation is visualized in the following diagram, where each column represents
a space generated by elementary sections (that is, a space isomorphic to $V^\alpha/V^{>\alpha}$).

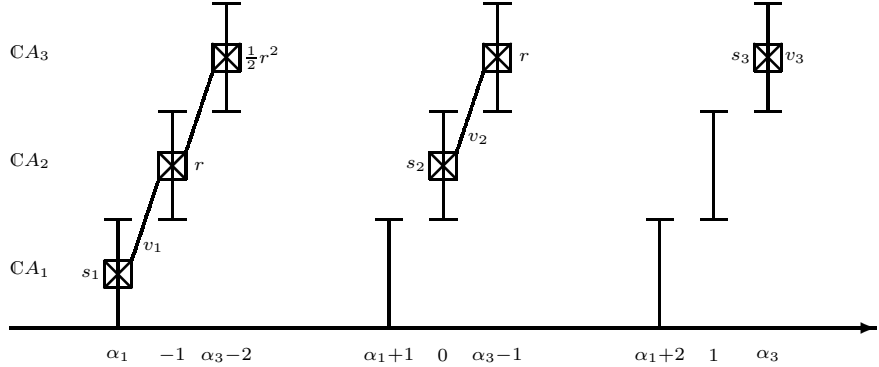
\begin{center}
\setlength{\unitlength}{0.36cm}
\begin{picture}(30,16)
\thicklines

\put(0,2){\vector(1,0){32}}

\put(0,4){\makebox{$\scriptstyle \dC A_1$}}
\put(0,8){\makebox{$\scriptstyle \dC A_2$}}
\put(0,12){\makebox{$\scriptstyle \dC A_3$}}

\multiput(4,2)(10,0){3}{\line(0,1){4}}

\multiput(6,6)(10,0){3}{\line(0,1){4}}

\multiput(8,10)(10,0){3}{\line(0,1){4}}


\multiput(3.5,6)(2,0){2}{\line(1,0){1}}
\multiput(5.5,10)(2,0){2}{\line(1,0){1}}

\multiput(13.5,6)(2,0){2}{\line(1,0){1}}
\multiput(15.5,10)(2,0){2}{\line(1,0){1}}

\multiput(23.5,6)(2,0){2}{\line(1,0){1}}
\multiput(25.5,10)(2,0){2}{\line(1,0){1}}

\multiput(7.5,14)(10,0){3}{\line(1,0){1}}

\put(3.5,3.5){\line(1,0){1}}\put(3.5,3.5){\line(0,1){1}}
\put(4.5,3.5){\line(0,1){1}}\put(3.5,4.5){\line(1,0){1}}
\put(3.5,3.5){\line(1,1){1}}\put(3.5,4.5){\line(1,-1){1}}
\put(3,4){\makebox(0,0){$\scriptstyle s_1$}}
\qbezier(4.5,4.5)(4.5,4.5)(5.5,7.5)

\multiput(5.5,7.5)(10,0){2}{\line(1,0){1}}\multiput(5.5,7.5)(10,0){2}{\line(0,1){1}}
\multiput(6.5,7.5)(10,0){2}{\line(0,1){1}}\multiput(5.5,8.5)(10,0){2}{\line(1,0){1}}
\multiput(5.5,7.5)(10,0){2}{\line(1,1){1}}\multiput(5.5,8.5)(10,0){2}{\line(1,-1){1}}
\put(15,8){\makebox(0,0){$\scriptstyle s_2$}}
\put(7,8){\makebox(0,0){$\scriptstyle r$}}
\qbezier(6.5,8.5)(6.5,8.5)(7.5,11.5)
\put(5.3,5){\makebox(0,0){$\scriptstyle v_1$}}
\put(17.3,9){\makebox(0,0){$\scriptstyle v_2$}}

\multiput(7.5,11.5)(20,0){2}{\line(1,0){1}}\multiput(7.5,11.5)(20,0){2}{\line(0,1){1}}
\multiput(8.5,11.5)(20,0){2}{\line(0,1){1}}\multiput(7.5,12.5)(20,0){2}{\line(1,0){1}}
\multiput(7.5,11.5)(20,0){2}{\line(1,1){1}}\multiput(7.5,12.5)(20,0){2}{\line(1,-1){1}}
\put(9.3,12){\makebox(0,0){$\scriptstyle \frac{1}{2}r^2$}}
\put(27,12){\makebox(0,0){$\scriptstyle s_3$}}
\put(29,12){\makebox(0,0){$\scriptstyle v_3$}}


\put(17.5,11.5){\line(1,0){1}}\put(17.5,11.5){\line(0,1){1}}
\put(18.5,11.5){\line(0,1){1}}\put(17.5,12.5){\line(1,0){1}}
\put(17.5,11.5){\line(1,1){1}}\put(17.5,12.5){\line(1,-1){1}}
\qbezier(16.5,8.5)(16.5,8.5)(17.5,11.5)
\put(19,12){\makebox(0,0){$\scriptstyle r$}}

\put(4,1){\makebox(0,0){$\scriptstyle \alpha_1$}}
\put(6,1){\makebox(0,0){$\scriptstyle -1$}}
\put(8,1){\makebox(0,0){$\scriptstyle \alpha_3-2$}}

\put(14,1){\makebox(0,0){$\scriptstyle \alpha_1+1$}}
\put(16,1){\makebox(0,0){$\scriptstyle 0$}}
\put(18,1){\makebox(0,0){$\scriptstyle \alpha_3-1$}}

\put(24,1){\makebox(0,0){$\scriptstyle \alpha_1+2$}}
\put(26,1){\makebox(0,0){$\scriptstyle 1$}}
\put(28,1){\makebox(0,0){$\scriptstyle \alpha_3$}}
\end{picture}
\end{center}

Let us calculate the variation of twistors associated to this example.
We have
$$
\begin{array}{rcl}
\tau v_1 &:= & s_3+\overline{r}zs_2+\frac{\overline{r}^2}{2}z^2s_1\\
\tau v_2 &:= & s_2+\overline{r}zs_1\\
\tau v_3 &:= & s_1,\\
\end{array}
$$
where $\tau$ is the morphism defined before Proposition \ref{propPropOfTau}.
A basis for the space $H^0(\dP^1,\widehat{\cH})$ is given for any $r$ with
$|r|\neq \sqrt{2}$ by
$w_1:=v_1$,
$w_2:=\overline{r}zs_1+(1+\frac12r\overline{r})s_2+$ $+rz^{-1}s_3$,
$w_3:=\tau v_1$.
The metric $h$ on $p_*\widehat{\cH}$ is in this basis
$H:=h(\underline{w}^{tr},\underline{w})=\diag\left((1-\frac12|r|^2)^2\right)$. This means that the space
of points $r\in \dC$ on which $\cH$ is pure
has two connected components, and that $\cH$ is pure polarized on each of them.

We are going to compute a limit TERP-structure, when the parameter $r$ tends to
infinity. First note that the following sections
are elements of $\cH$: $zs_1+\frac{r}{2}s_2=zv_1-\frac{r}{2}v_2$ and
$z^2s_1=z^2v_1-rzv_2+\frac{r^2}{2}v_3$.
Taking the basis $(\widetilde{v}_1,\widetilde{v}_2,\widetilde{v}_3)$, where
$$
\begin{array}{rcl}
\widetilde{v}_1 & = & (r')^2s_1+ r' z^{-1} s_2 +\frac12 z^{-2}s_3\\
\widetilde{v}_2 & = & r' zs_1 + \frac12 s_2\\
\widetilde{v}_3 & = & z^2 s_1
\end{array}
$$
and $r':=r^{-1}$, of $\cH$ outside of $r=0$ shows
that we obtain an extension of the family to $\dP^1$.
We see (by setting $r':=0$) that the fibre $\cH_{|r=\infty}$ is the TERP-structure
generated by the elementary sections $z^{-2}s_3, s_2, z^2s_1$. In particular,
the spectral numbers, which are constant for finite $r$ (namely, $\alpha_1,\alpha_2,\alpha_3$)
have changed to $\alpha_3-2,\alpha_2, \alpha_1+2$.

One might analyze this example further by calculating the classifying
space $D_{BL}$ attached to the initial data used above, namely, for the
spectrum $\alpha_1,\alpha_2,\alpha_3$. We give the result, without carrying out
the details.
\begin{lemma}
For the given initial data from above, we have that $D_{BL}\cong\dC^2$, with
the universal family of Brieskorn lattices given by $\cH=\oplus_{i=1}^3\cO_{\dC^3}v_i$,
where
$$
\begin{array}{rcl}
v_1 & := & s_1 + r z^{-1} s_2 + \frac{r^2}{2} z^{-2} s_3 + t z^{-1}s_3\\
v_2 & := & s_2 + r z^{-1} s_3\\
v_3 & := & s_3.
\end{array}
$$
The family of twistors $\widehat{\cH}$ is pure outside of the real-analytic hypersurface
$(1-\rho)^4-\theta=0$, where $\rho=\frac12r\overline{r}$ and $\theta=t\overline{t}$.

The complement has three components. The family of twistors is polarized on two of them,
those which contain $\{(r,0)\ |\ |r|<\sqrt{2}\}$ and respectively $\{(r,0)\ |\ |r|>\sqrt{2}\}$.
On the third component the metric on $p_*\widehat \cH$ has signature $(+,-,-)$.
The restriction of the metric on $\cT_{\cblp}$ to the tangent space of $\{(r,0)\ |\ |r|\neq \sqrt{2}\}$
is given by
$$
h(\partial_r,\partial_r) = 2\frac{1+\rho^2}{(1-\rho)^4}.
$$
\end{lemma}

So we see that the above family $\cH(r)$ is in fact the restriction of
the universal bundle on $D_{BL}$ to the subspace $t=0$. The fact that the
family $\cH(r)$ does not have a limit inside $D_{BL}$ (and not
even inside $\cbl$, which is equal to $D_{BL}$ in this case)
when $r\rightarrow \infty$ has the consequence that
the metric on $\cT_{\cblp}$ is not complete. This can be seen directly,
namely, we have
$$
h(\partial_{r'},\partial_{r'})=h(-r^2\partial_r,-r^2\partial_r)
=8\rho^2\frac{1+\rho^2}{(1-\rho)^4}
\stackrel{r\rightarrow \infty}{\longrightarrow} 8.
$$
\end{example}
For applications, it is important that the target domain of the
period map is a complete Hermitian space. For that reason, we
are going to describe a partial compactification of $\cblp$ on which we
are able to construct such a complete Hermitian metric.

The main idea is to relax slightly the initial data to be fixed by
requiring only that the spectral numbers be contained in an interval.
More precisely, we fix as before
$H^\infty,H^\infty_\dR, S, M, w$ and a rational number
$\alpha_1$ (for notational convenience, we denote $\alpha_\mu:=w-\alpha_1$).
\begin{theorem}[\cite{HS3}, Section 7]
There is a projective variety $\MBL$ containing $\check{D}_{BL}$ as a locally closed subset
and which comes equipped with a universal locally free sheaf $\cH$ whose fibres
are regular singular TERP-structures with spectral numbers in $[\alpha_1,\alpha_\mu]$ and initial
data $(H^\infty,H^\infty_\dR,S,M)$. On $\cbl$ it reduces to the
universal bundle considered above. $\MBL$ is stratified by locally closed subsets
of regular singular TERP-structures with fixed spectral pairs.
\end{theorem}
\begin{remark}
We will give some remarks on the construction of $\MBL$ and the proof of the theorem. The basic idea
is that by fixing the range of spectral numbers to lie in
the interval $[\alpha_1,\alpha_{\mu}]$, one
is able to reduce the setting to a finite-dimensional situation and to control the possible
regular singular TERP-structures with spectral numbers in that range.
More precisely, the pairing $P$ induces a symplectic
structure on the quotient $V^{\alpha_1}/V^{>\alpha_\mu-1}$ and the operators $z\cdot $ and $z^2\nabla_z$ give rise
to nilpotent endomorphisms of $V^{\alpha_1}/V^{>\alpha_\mu-1}$. Then we consider the subvariety
of the Lagrangian Grassmannian of $V^{\alpha_1}/V^{>\alpha_\mu-1}$ consisting of subspaces $K$ which are
invariant under these two nilpotent endomorphisms. For any regular singular
TERP-structure $H$ the space
$K:=H/V^{>\alpha_\mu-1}$ gives a point in this subvariety and we obtain a bijective correspondence
between regular singular TERP-structures (with spectral numbers in $[\alpha_1,\alpha_\mu]$)
and these subspaces $K$. It is clear that this subvariety is projective, but it might be very
singular, and even non-reduced. We define $\MBL$ to be this variety with its reduced scheme structure.
By construction $\MBL$ carries a universal sheaf of Brieskorn lattices, with an induced V-filtration, so
that fixing the spectral pairs gives an intersection of open and closed conditions. This yields a
stratification where each space $\cblp$ (for fixed spectral pairs such that
its spectral numbers are contained in $[\alpha_1,\alpha_\mu]$) is a locally closed stratum.
\end{remark}

Let
$$
\Theta:=\left\{x\in \MBL\,|\,\widehat{\cH} \mbox{ not pure }\right\}.
$$
It is clear that $\Theta$ is a real analytic subvariety of $\MBL$. The complement
$\MBL\backslash \Theta$ splits into connected components, on which the
form $h:p_*\widehat{\cH}\otimes\overline{p_*\widehat{\cH}}\rightarrow \cC^\infty_{\MBL}$ is defined.
Consider the union of those components on which it is positive definite, denote this
space by $\MBLp$,
then we have the following, which is one of the main results in \cite{HS3}.
\begin{theorem}[\cite{HS3}, Theorem 8.6]
\label{theoCompleteMetric}
The positive definite Hermitian metric $h$ on $(p_*\widehat{\cH})_{|\MBLp}$ induces
a \textbf{complete} distance $d_h$ on $\MBLp$, which defines its topology.
The same holds for the closed subspace
$\overline{\check{D}}_{BL}\cap \MBLp\subset \MBLp$, which therefore is a partial compactification of
$\cblp$ with complete distance.
\end{theorem}
If the monodromy respects a lattice in $H^\infty_\dR$, we have the following.
\begin{proposition}[\cite{HS3}, Theorem 8.8]
Suppose that $H^\infty_\dZ\subset H^\infty_\dR$ is a lattice with
$M\in\Aut(H^\infty_\dZ)$. Then the discrete group $G_\dZ:=\Aut(H^\infty_\dZ, S, M)$
acts properly discontinuously on $\MBLp$and on $\overline{\check{D}}_{BL}\cap \MBLp$
so that the quotients $\MBLp/G_\dZ$ and $(\overline{\check{D}}_{BL}\cap
\MBLp)/G_\dZ$ are complex spaces.
\end{proposition}

In the following proposition, we describe the structure of the
space $\MBL$ for the example considered in the beginning of this section.
\begin{proposition}
The space $\MBL$ for the initial data from example \ref{exP114} is the
weighted projective space $\dP(1,1,4)$, where the embedding
$\cbl \hookrightarrow \dP(1,1,4)=\textup{Proj}\,\dC[X,Y,Z]$
is given by $X:=r, Y:=1,Z:=t$ (here $\deg(X)=1, \deg(Y)=1, \deg(Z)=4$).
The only singular point of $\dP(1,1,4)$ is $(0:0:1)$,
in particular, $\MBLp$ is smooth in this case.
\end{proposition}

The following is a direct application of Theorem \ref{theoCompleteMetric}.
\begin{theorem}[\cite{HS3}, Theorem 9.5]
Let $X$ be a complex variety, $Z\subset X$ a complex space of codimension at least two. Suppose
that the complement $Y:=X\backslash Z$ is simply connected. Let
$\TERP$ be a variation of regular singular pure polarized TERP-structures on the complement $Y$ which has
constant spectral pairs. Then this variation extends to the whole of $X$, with possibly
jumping spectral numbers over $Z$.
\end{theorem}
This theorem can be proved exactly as in the case of Hodge structures, namely, the given variation
defines a period map $Y\rightarrow \cblp$. In particular, due to
Corollary \ref{corAhlfors}, we know that this map is distance decreasing with respect to the
(complete) distance $d_h$ on $\overline{\check{D}}_{BL}\cap \MBLp$
and the Kobayashi pseudometric on $Y$. This implies that it extends
continuously and then also holomorphically to $X$ (the proof is the same as in \cite[Corollary 3.5]{Ko2},
note however that $\overline{\check{D}}_{BL}\cap \MBLp$
is not a hyperbolic space, but the period map is horizontal by assumption, which is sufficient).

To finish this section, we describe two further examples of classifying spaces $\MBL$.
More examples can be found in \cite{HS3}.
\begin{example}[\cite{HS3}, Section 9.3]
Here is an example where $\MBL$ is reducible. Fix any number $n\in\dN_{>0}$ and
consider the initial data:
$H^\infty_\dR:=\dR B_1 \oplus \dR B_2$, $M:={\rm id}\in\Aut(H^\infty_\dR)$,
and $S(\underline{B}^{tr},\underline{B})=(-1)^n\frac12\diag(1,1)$. Let $A_1:=B_1+iB_2$, $A_2:=\overline{A}_1=B_1-iB_2$
and consider the following reference filtration
$$
\{0\}= F_0^{n+1} \subsetneq F_0^n:=\dC A_1 = F_0^{n-1} = \cdots = F_0^{-n+1}\subsetneq F_0^{-n} = H^\infty
$$
Then $\alpha_1=-n$ and
the classifying space is $\DcPMHS=\DPMHS=\DcPHS=\DPHS$,
it consists of two points, namely $F^\bullet_0$ and
$\overline{F}^\bullet_0$, which are both pure polarized Hodge structures of
weight zero with Hodge decomposition $H^{n,-n}\oplus H^{-n,n}$.
Put $s_1:=z^{-n}A_1$ and $s_2:=z^n A_2$.
The classifying space $D_{BL}$ is a disjoint union of two affine lines, namely, the universal family
over the component above $F_0^\bullet$ is given by
$\cH_{-n}(r):=\cO_{\dC^2} v_1\oplus \cO_{\dC^2} v_2$, where
$v_1:=z^{-n}A_1+rz^{n-1}A_2=s_1+rz^{-1}s_2, v_2:=z^n A_2=s_2$, and
the universal family over the other component is given by
$\cH_n(r):=\cO_{\dC^2} (z^{-n}A_2+rz^{n-1}A_1)\oplus \cO_{\dC^2} z^n A_1$.
The following diagram visualizes this situation.
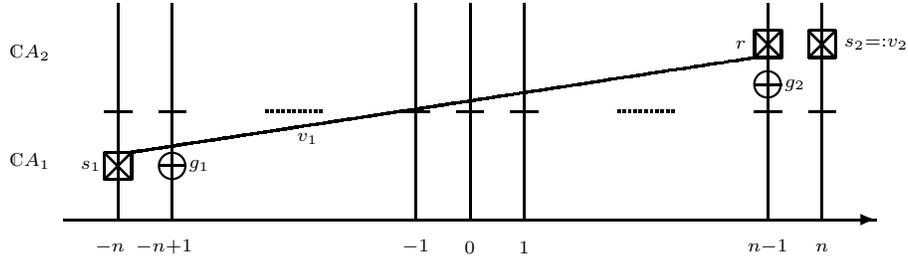
\begin{center}
\setlength{\unitlength}{0.36cm}
\begin{picture}(30,12)
\thicklines

\put(0,2){\vector(1,0){30}}

\put(-2,4){\makebox{$\scriptstyle \dC A_1$}}
\put(-2,8){\makebox{$\scriptstyle \dC A_2$}}

\multiput(2,2)(2,0){2}{\line(0,1){8}}
\multiput(13,2)(2,0){3}{\line(0,1){8}}
\multiput(13,2)(2,0){3}{\line(0,1){8}}
\multiput(26,2)(2,0){2}{\line(0,1){8}}

\multiput(1.5,6)(2,0){2}{\line(1,0){1}}
\multiput(12.5,6)(2,0){3}{\line(1,0){1}}
\multiput(25.5,6)(2,0){2}{\line(1,0){1}}
\multiput(7.5,6)(13,0){2}{\dashbox{0.1}(2,0.02){}}

\put(1.5,3.5){\line(1,0){1}}\put(1.5,3.5){\line(0,1){1}}
\put(2.5,3.5){\line(0,1){1}}\put(1.5,4.5){\line(1,0){1}}
\put(1.5,3.5){\line(1,1){1}}\put(1.5,4.5){\line(1,-1){1}}

\put(25.5,8){\line(1,0){1}}\put(25.5,8){\line(0,1){1}}
\put(26.5,8){\line(0,1){1}}\put(25.5,9){\line(1,0){1}}
\put(25.5,8){\line(1,1){1}}\put(25.5,9){\line(1,-1){1}}

\put(4,4){\circle{1}}\put(3.5,4){\line(1,0){1}}
\put(26,7){\circle{1}}\put(25.5,7){\line(1,0){1}}

\put(27.5,8){\line(1,0){1}}\put(27.5,8){\line(0,1){1}}
\put(28.5,8){\line(0,1){1}}\put(27.5,9){\line(1,0){1}}
\put(27.5,8){\line(1,1){1}}\put(27.5,9){\line(1,-1){1}}

\qbezier(2.5,4.5)(2.5,4.5)(25.5,8)

\put(1,4){\makebox(0,0){$\scriptstyle s_1$}}
\put(9,5){\makebox(0,0){$\scriptstyle v_1$}}
\put(5,4){\makebox(0,0){$\scriptstyle g_1$}}

\put(25,8.5){\makebox(0,0){$\scriptstyle r$}}
\put(30,8.5){\makebox(0,0){$\scriptstyle s_2=:v_2$}}
\put(27,7){\makebox(0,0){$\scriptstyle g_2$}}

\put(1.7,1){\makebox(0,0){$\scriptstyle -n$}}
\put(3.7,1){\makebox(0,0){$\scriptstyle -n+1$}}
\put(13,1){\makebox(0,0){$\scriptstyle -1$}}
\put(15,1){\makebox(0,0){$\scriptstyle 0$}}
\put(17,1){\makebox(0,0){$\scriptstyle 1$}}
\put(26,1){\makebox(0,0){$\scriptstyle n-1$}}
\put(28,1){\makebox(0,0){$\scriptstyle n$}}
\end{picture}
\end{center}
It is directly evident that the ``limit TERP''-structure
(when $r$ approaches infinity), is given by $\cG_{-n+1}:=\cO_\dC g_1\oplus
\cO_\dC g_2$, where $g_1:=z s_1$ and $g_2:=z^{-1}s_2$. We see that $\cG_{-n+1}$
is the origin in one of the two components of the classifying space $\cbl$
associated to the same data, but with $\alpha_1:=-n+1$. Note, however, that
the two (conjugate) filtrations induced by $\cG_{-n+1}$ and $\cG_{n-1}$, respectively,
are not pure polarized: the Hodge metric is negative definite.
Taking the limit of the universal family for this classifying spaces yields
TERP-structures $\cG_{-n+2}$ and $\cG_{n-2}$, respectively, and we can continue
this procedure until we arrive at $\cG_{-1}$ and $\cG_1$. The limits
$\lim_{r\rightarrow \infty}\cH_{-1}(r)$ and $\lim_{r\rightarrow \infty}\cH_1(r)$ are
both equal to the lattice $\cG_0=   V^0$.
This shows that the space $\MBL$ is a chain of $2n$ copies of $\dP^1$, where the Hodge filtration
gives pure polarized, resp. negative definite,
pure Hodge structures on every other component of this chain.
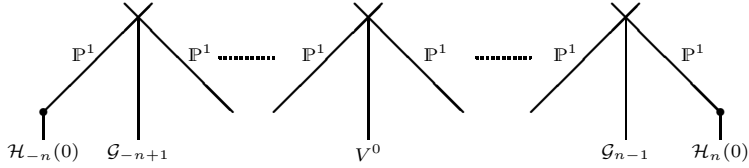
\begin{center}
\setlength{\unitlength}{0.36cm}
\begin{picture}(29,8)
\thicklines
\multiput(2,2)(18,0){2}{\line(1,1){4}}
\multiput(3.5,4.2)(22.5,0){2}{\makebox(0,0){$\scriptstyle \dP^1$}}
\multiput(5,6)(18,0){2}{\line(1,-1){4}}
\multiput(7.8,4.2)(13.5,0){2}{\makebox(0,0){$\scriptstyle \dP^1$}}
\put(10.5,2){\line(1,1){4}}
\put(12,4.2){\makebox(0,0){$\scriptstyle \dP^1$}}
\put(13.5,6){\line(1,-1){4}}
\put(16.5,4.2){\makebox(0,0){$\scriptstyle \dP^1$}}

\multiput(8.5,4)(9.5,0){2}{\dashbox{0.1}(2,0.02){}}
\thinlines
\qbezier(14,5.5)(14,5.5)(14,1)
\put(14,0.5){\makebox(0,0){$\scriptstyle V^{\scriptscriptstyle 0}$}}

\qbezier(5.5,5.5)(5.5,5.5)(5.5,1)
\put(5.5,0.5){\makebox(0,0){$\scriptstyle \cG_{\scriptscriptstyle -n+1}$}}

\qbezier(23.5,5.5)(23.5,5.5)(23.5,1)
\put(23.5,0.5){\makebox(0,0){$\scriptstyle \cG_{\scriptscriptstyle n-1}$}}

\put(2,2){\circle*{0.3}}
\qbezier(2,2)(2,2)(2,1)
\put(2,0.5){\makebox(0,0){$\scriptstyle \cH_{\scriptscriptstyle -n}(0)$}}
\put(27,2){\circle*{0.3}}
\qbezier(27,2)(27,2)(27,1)
\put(27,0.5){\makebox(0,0){$\scriptstyle \cH_{\scriptscriptstyle n}(0)$}}
\end{picture}
\end{center}

It is easy to calculate the associated twistors: For the original family $\cH_{-n}$,
we have
$$
\widehat{\cH}_{-n}:=\cO_{\dP^1}\cC^\infty_{\dC}(\underbrace{s_1+r z^{-1} s_2}_{w_1})\oplus \cO_{\dP^1}\cC^\infty_{\dC}(\underbrace{s_2+\overline{r} z s_1}_{w_2})
$$
and the metric is $h(w_1,w_2)=\diag(1-|r|^2)$ (note that $P(s_1,s_2)=1$ due to our choices), so that $\cH_{-n}(r)$ is a variation
of pure polarized TERP-structures on $\Delta^*$ (it is even a nilpotent orbit, which corresponds by
Theorem \ref{theoCorrMHSNilpot} to the pure polarized Hodge structure $(H^\infty,H^\infty_\dR, S, F^\bullet_0)$).
The same holds for the family $\cH_n(r)$.

Note however that due to $P(g_1,g_2)=-1$, the variation of twistors on the second left-(or right-)most $\dP^1$
is pure polarized on $\dP^1\backslash\overline{\Delta}$, where the origin is
the TERP-structure $\cG_{-n+1}$ (resp. $\cG_{n-1}$).
This means that in the above picture, the points of intersection on the lower level are pure polarized,
but not those on the upper level. In particular, $V^0$ is pure polarized precisely if $n$ is even (the above picture
already supposes that $n$ is odd), which can also be seen directly from the formula $S(A_1,A_2)=(-1)^n$.
\end{example}

\begin{example}[\cite{HS3}, Section 9.1]
This example shows that $\MBL$ is not necessarily a union or a product of projective or weighted projective
spaces. Consider the following initial data:
$H_\dR^\infty:=\oplus_{i=1}^4 \dR B_i$, $A_1:=B_1+iB_4$,
$A_2:=B_2+iB_3$, $\overline{A}_1=A_4$, $\overline{A}_2=A_3$,
$M(A_i):=e^{-2\pi i \alpha_i}A_i$, where we choose $\alpha_1$ and $\alpha_2$
with $-1<\alpha_1<\alpha_2<-\frac12$ and $\alpha_3=-\alpha_2$, $\alpha_4=-\alpha_1$. Moreover, we put
$$
S(\underline{B}^{tr},\underline{B}):=
\begin{pmatrix}
0 & 0 & 0 & -\gamma_1 \\
0 & 0 & -\gamma_2 & 0\\
0 & \gamma_2 & 0 & 0 \\
\gamma_1 & 0 & 0 & 0
\end{pmatrix}
$$
where $\gamma_1:=\frac{-1}{4\pi}\Gamma(\alpha_1+1)\Gamma(\alpha_4)$ and
$\gamma_2:=\frac{-1}{4\pi}\Gamma(\alpha_2+1)\Gamma(\alpha_3)$
(so that $S(A_1,A_4)=2i\gamma_1$ and $S(A_2,A_3)=2i\gamma_2$).
Let $s_i:=z^{\alpha_i}A_i$ and $\cH:=\oplus_{i=1}^4\cO_{\dC^4}v_i$, where
$$
\begin{array}{rcl}
v_1 & := & s_1 + r z^{-1} s_3 + q z^{-1} s_4\\
v_2 & := & s_2 + p z^{-1} s_3 + r z^{-1} s_4\\
v_3 & := & s_3 \\
v_4 & := & s_4
\end{array}
$$
Then $\TERP$ is a variation of TERP-structures of weight zero on $\dC^3$ with constant spectrum
$\Sp=(\alpha_1,\ldots,\alpha_4)$,
it is in fact the universal family of the classifying space $\cbl$ associated
to the initial data $(H^\infty,H^\infty_\dR, S, M, \Sp)$. The induced (constant) filtration is
$$
\{0\}=F^1\subsetneq F^0:=\dC A_1 \oplus \dC A_2 \subsetneq F^{-1} = H^\infty
$$
from which one checks that $(H^\infty, H_\dR^\infty, S, F^\bullet)$ is a pure polarized Hodge structure
of weight $-1$. The associated diagram looks as follows.
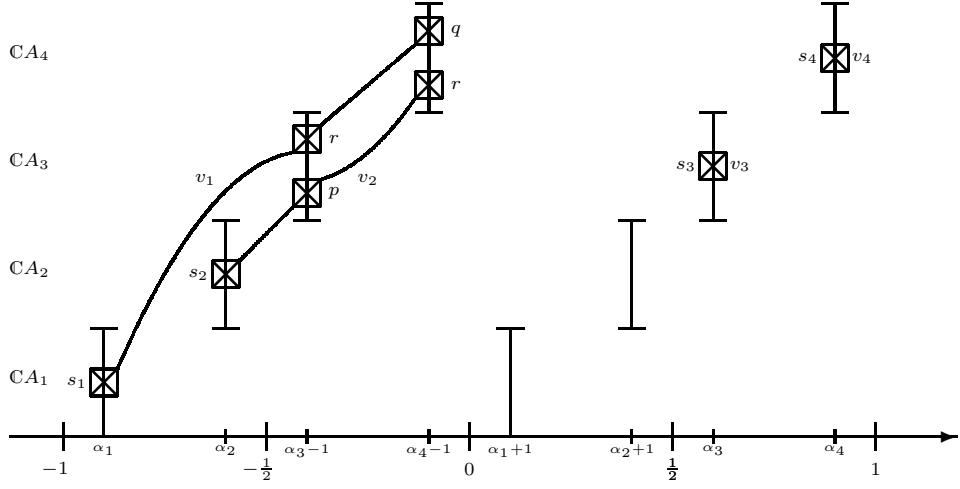
\begin{center}
\setlength{\unitlength}{0.36cm}
\begin{picture}(31,20)
\thicklines

\put(-2,2){\vector(1,0){35}}

\put(-2,4){\makebox{$\scriptstyle \dC A_1$}}
\put(-2,8){\makebox{$\scriptstyle \dC A_2$}}
\put(-2,12){\makebox{$\scriptstyle \dC A_3$}}
\put(-2,16){\makebox{$\scriptstyle \dC A_4$}}

\put(0,1.5){\line(0,1){1}}
\put(-0.3,0.8){\makebox(0,0){$\scriptstyle -1$}}
\put(7.5,1.5){\line(0,1){1}}
\put(7.2,0.8){\makebox(0,0){$\scriptstyle -\frac12$}}
\put(15,1.5){\line(0,1){1}}
\put(15,0.8){\makebox(0,0){$\scriptstyle 0$}}
\put(22.5,1.5){\line(0,1){1}}
\put(22.5,0.8){\makebox(0,0){$\scriptstyle \frac12$}}
\put(22.5,1.5){\line(0,1){1}}
\put(22.5,0.8){\makebox(0,0){$\scriptstyle \frac12$}}
\put(30,1.5){\line(0,1){1}}
\put(30,0.8){\makebox(0,0){$\scriptstyle 1$}}

\multiput(1.5,2)(15,0){2}{\line(0,1){4}}
\put(1.5,1.5){\makebox(0,0){$\scriptscriptstyle \alpha_1$}}
\put(16.5,1.5){\makebox(0,0){$\scriptscriptstyle \alpha_1+1$}}
\multiput(1,6)(15,0){2}{\line(1,0){1}}
\multiput(6,6)(15,0){2}{\line(0,1){4}}\multiput(6,1.75)(15,0){2}{\line(0,1){0.5}}
\put(6,1.5){\makebox(0,0){$\scriptscriptstyle \alpha_2$}}
\put(21,1.5){\makebox(0,0){$\scriptscriptstyle \alpha_2+1$}}
\multiput(5.5,6)(15,0){2}{\line(1,0){1}}
\multiput(5.5,10)(15,0){2}{\line(1,0){1}}

\multiput(9,10)(15,0){2}{\line(0,1){4}}
\multiput(9,1.75)(15,0){2}{\line(0,1){0.5}}
\put(9,1.5){\makebox(0,0){$\scriptscriptstyle \alpha_3-1$}}
\put(24,1.5){\makebox(0,0){$\scriptscriptstyle \alpha_3$}}
\multiput(8.5,10)(15,0){2}{\line(1,0){1}}
\multiput(8.5,14)(15,0){2}{\line(1,0){1}}

\multiput(13.5,14)(15,0){2}{\line(0,1){4}}
\multiput(13.5,1.75)(15,0){2}{\line(0,1){0.5}}
\put(13.5,1.5){\makebox(0,0){$\scriptscriptstyle \alpha_4-1$}}
\put(28.5,1.5){\makebox(0,0){$\scriptscriptstyle \alpha_4$}}
\multiput(13,14)(15,0){2}{\line(1,0){1}}
\multiput(13,18)(15,0){2}{\line(1,0){1}}

\multiput(1,3.5)(4.5,4){2}{\line(1,0){1}}\multiput(1,3.5)(4.5,4){2}{\line(0,1){1}}
\multiput(2,3.5)(4.5,4){2}{\line(0,1){1}}\multiput(1,4.5)(4.5,4){2}{\line(1,0){1}}
\multiput(1,3.5)(4.5,4){2}{\line(1,1){1}}\multiput(1,4.5)(4.5,4){2}{\line(1,-1){1}}
\qbezier(2,4.5)(5.25,12)(8.5,12.5)
\qbezier(9.5,13.5)(9.5,13.5)(13,16.5)
\qbezier(6.5,8.5)(6.5,8.5)(8.5,10.5)
\qbezier(9.5,11.5)(11.25,12)(13,14.5)
\put(0.5,4){\makebox(0,0){$\scriptstyle s_1$}}
\put(5,8){\makebox(0,0){$\scriptstyle s_2$}}
\put(5.25,11.5){\makebox(0,0){$\scriptstyle v_1$}}
\put(11.25,11.5){\makebox(0,0){$\scriptstyle v_2$}}

\multiput(8.5,10.5)(4.5,4){2}{\line(1,0){1}}\multiput(8.5,10.5)(4.5,4){2}{\line(0,1){1}}
\multiput(9.5,10.5)(4.5,4){2}{\line(0,1){1}}\multiput(8.5,11.5)(4.5,4){2}{\line(1,0){1}}
\multiput(8.5,10.5)(4.5,4){2}{\line(1,1){1}}\multiput(8.5,11.5)(4.5,4){2}{\line(1,-1){1}}
\put(10,11){\makebox(0,0){$\scriptstyle p$}}
\put(10,13){\makebox(0,0){$\scriptstyle r$}}
\put(14.5,15){\makebox(0,0){$\scriptstyle r$}}
\put(14.5,17){\makebox(0,0){$\scriptstyle q$}}

\multiput(8.5,12.5)(4.5,4){2}{\line(1,0){1}}\multiput(8.5,12.5)(4.5,4){2}{\line(0,1){1}}
\multiput(9.5,12.5)(4.5,4){2}{\line(0,1){1}}\multiput(8.5,13.5)(4.5,4){2}{\line(1,0){1}}
\multiput(8.5,12.5)(4.5,4){2}{\line(1,1){1}}\multiput(8.5,13.5)(4.5,4){2}{\line(1,-1){1}}

\multiput(23.5,11.5)(4.5,4){2}{\line(1,0){1}}\multiput(23.5,11.5)(4.5,4){2}{\line(0,1){1}}
\multiput(24.5,11.5)(4.5,4){2}{\line(0,1){1}}\multiput(23.5,12.5)(4.5,4){2}{\line(1,0){1}}
\multiput(23.5,11.5)(4.5,4){2}{\line(1,1){1}}\multiput(23.5,12.5)(4.5,4){2}{\line(1,-1){1}}
\put(23,12){\makebox(0,0){$\scriptstyle s_3$}}
\put(27.5,16){\makebox(0,0){$\scriptstyle s_4$}}
\put(25,12){\makebox(0,0){$\scriptstyle v_3$}}
\put(29.5,16){\makebox(0,0){$\scriptstyle v_4$}}
\end{picture}
\end{center}
Let us calculate the space $\MBL$ associated to the
initial data $(H^\infty, H^\infty_\dR, S, M, \alpha_1, w)$.
It is very simple in this example, as we have $\lfloor\alpha_\mu-\alpha_1\rfloor=1$.
This implies that $zv_i\in V^{>\alpha_4-1}$ and $(z^2\nabla_z)v_i\in V^{>\alpha_4-1}$ for $i=1,\ldots,4$,
so that
$$
\MBL=\textup{LGr}(V^{\alpha_1}/V^{>\alpha_4-1}, [P^{(-1)}]),
$$
where $[P^{(-1)}]=
(z^{-1}s_4)^*\wedge (s_1)^* + (z^{-1}s_3)^*\wedge (s_2)^* \in \bigwedge^2 (V^{\alpha_1}/V^{>\alpha_4-1})^*$
is the symplectic form induced by $P$ and $\textup{LGr}(V^{\alpha_1}/V^{>\alpha_4-1}, [P^{(-1)}])$
denotes the Lagrangian Grassmannian of subspaces of dimension two of
$V^{\alpha_1}/V^{>\alpha_4-1}$ on which $[P^{(-1)}]$ vanishes. Using
the Pl{\"u}cker embedding, one checks that this Lagrangian Grassmannian
is a hyperplane section of the Pl\"ucker quadric in $\dP^5$, i.e., a smooth quadric in $\dP^4$
which is isomorphic neither to $\dP^3$ nor to $\dP^1\times\dP^2$.
It is also clear that $\MBL$ is indeed the closure of the three-dimensional
affine classifying space $\check{D}_{BL}$ considered above, but it also contains other
strata, e.g., the closure of the two-dimensional
classifying space associated to (the same topological data and) the spectrum
$\Sp=(\alpha_1,\alpha_3-1, \alpha_2+1, \alpha_4)$ is the weighted projective
space $\dP^1(1,1,2)$, which appears as a (further) hyperplane section of the
Lagrangian Grassmannian $\MBL$.
\end{example}

\section{Period mappings and limit structures}
\label{secLimit}

In this section, we will give an interpretation of one of the main
results of \cite{Mo2} to the case of period mappings associated to isolated hypersurface
singularities. We are considering variations of TERP-structures on
non-compact varieties, and we are interested in studying their
behaviour at the boundary. In order to keep the notation as simple as
possible, we will restrict here to variations on a punctured disc. This
is sufficient for many applications; the details for the higher
dimensional case are worked out in \cite{HS3}.
We will denote throughout this section by $r$ a coordinate on $\Delta$.
\begin{lemma}[\cite{HS3}, Sections 3 and 5]
Let $H$ be a variation of pure polarized regular singular TERP-structures on $\Delta^*$. Then
for any $a\in\dR$, there are vector bundles $_a\cE$ (resp. ${_{<a}}\cE$) on $\dC\times \Delta$ which extend
$H$ and such that the restriction $(_a\cE)_{|\dC^*\times \Delta}$ (resp.
$(_{<a}\cE)_{|\dC^*\times \Delta}$) coincides with
the Deligne extension $\cV_\Delta^{-a}$ (resp. $\cV_\Delta^{>-a}$ ) of order $-a$ (resp. bigger than $-a$)
of $H'$ over $\dC^*\times \Delta$. Moreover, we have a connection
$$
\nabla:{_a}\cE\longrightarrow {_a}\cE\otimes z^{-1}\Omega^1_{\dC\times\Delta}(\log(\{0\}\times \Delta\cup\dC\times\{0\}))
$$
extending the given connection operators on $\cH$ and $\cV_\Delta^{-a}$. There is a non-degenerate pairing
$$
P:{_a}\cE\otimes{_{<1-a}}\cE\longrightarrow z^w\cO_{\dC\times\Delta}
$$
which restricts to the given $P$ on $\dC\times\Delta^*$.
\end{lemma}
As a consequence, we can construct a limit object which turns out to be a TERP-structure.
\begin{theorem}[\cite{HS3}, Theorem 3.5]
\label{theoLimitTERP}
Let $H$ be as above. Then the $\dC$-bundle
$$
\cG:=\bigoplus_{a\in(0,1]} {_a}\cE/{_{<a}}\cE
$$
underlies a regular singular TERP-structure. It comes equipped with a nilpotent endomorphism
$$
\cN:\cG\longrightarrow z^{-1}\cG
$$
given by the residue of the operator $r\nabla_r$ on $\cG$ along $\dC\times\{0\}$.
\end{theorem}
In \cite{Mo2}, a construction for a limit object starting from
a general tame variation of pure polarized twistor structures is described.
He proves that this is a polarized mixed twistor structure.
It turns out that this limit can be identified with the twistor $\widehat{\cG}$
constructed from the TERP-structure $\cG$.
\begin{theorem}[\cite{HS3}, Theorem 3.7]
The $\dP^1$-bundle $\widehat{\cG}$ underlies a polarized mixed twistor structure.
\end{theorem}
We refer to \cite{Mo2} or \cite{HS1} for a precise definition of a polarized mixed twistor structure.
It is essentially given by the additional datum of a weight filtration by subbundles $\widehat{\cW}_k$ of $\widehat{\cG}$
such that the quotients are pure twistors of weight $k$ (i.e. isomorphic to a $\cO^n_{\dP^1}(k)$ for some $n\in\dN$)
and such that a positivity condition is satisfied on primitive subspaces.
In our case $\widehat{\cW}_\bullet$ is defined by the weight filtration $\cW_\bullet$ on $\cG$
associated to $\cN$.

A very easy but important consequence is the following, which is the analogue of
\cite[Corollary 4.11]{Sch}.
\begin{corollary}
Let $H$ be as above, and suppose that the monodromy around $\dC\times\Delta^*$ of the local system
$H'$ is trivial. Then for any $a\in \dR$, ${_a}\cE$ underlies a variation of pure polarized regular singular
TERP-structures on $\Delta$, extending $H$.
The limit $\cG$ is naturally identified with the fibre $_a\cE/r\cdot{_a}\cE$.
In particular, $\cG$ is still pure polarized.
\end{corollary}
We are mainly interested in TERP-structures defined by isolated hypersurface singularities.
In that case, we obtain the following result.
\begin{theorem}[\cite{HS3}, Theorem 9.7]
Given a $\mu$-constant family of isolated hypersurface singularities
$F:(\dC^{n+1}\times \Delta^*,0)\rightarrow (\dC,0)$, suppose that
the variation $\mathit{TERP}(F)$ over $\Delta^*$ is pure polarized.
\begin{itemize}
\item
If the monodromy around $\dC\times\{0\}$ is semi-simple, then the
locally liftable period map
$\phi_{BL}:\Delta^*\rightarrow \cblp/G_\dZ\subset \overline{\check{D}}_{BL}\cap\MBLp/G_\dZ$
extends holomorphically
(not necessarily locally liftable) to $\phi_{BL}:\Delta\rightarrow \overline{\check{D}}_{BL}\cap \MBLp/G_\dZ$.
\item
Suppose that the monodromy around $\dC\times\{0\}$ is reduced to the identity,
in particular, there is a period map
$\phi_{BL}:\Delta^*\rightarrow \cblp\subset \overline{\check{D}}_{BL}\cap \MBLp$. Then this map extends to
$$
\overline{\phi}:\Delta\rightarrow \overline{\check{D}}_{BL}\cap \MBLp
$$
and the variation $\mathit{TERP}(F)$ extends to a variation on $\Delta$.
\end{itemize}
\end{theorem}

We finish this survey by a result concerning the
``new supersymmetric index'', i.e., the endomorphism $\cQ$ defined at the end of
chapter \ref{secTwistor}.

\begin{theorem}[\cite{He4}, Theorem 7.20]
Consider a nilpotent orbit of regular singular TERP-structures
$K$, i.e., a family over $\Delta^*$ (coordinate $r$),
constructed as $K:=\pi^*H$ from a single TERP-structure $H$,
where $\pi:\dC\times\dC^*\to\dC$, $(z,r)\mapsto
zr$. By definition $K$ is pure polarized for small
$|r|$, and we have the Hermitian endomorphism $\cQ$ of
$p_*\widehat{\cK}$. Then its eigenvalues tend for $r\rightarrow
0$ to $\frac{w}{2}-\Sp(H,\nabla)$.
\end{theorem}
A generalization of this result to the irregular case
can be found in the recent preprint \cite{Sa10}.

%

\bibliographystyle{amsalpha}
\providecommand{\bysame}{\leavevmode\hbox to3em{\hrulefill}\thinspace}
\providecommand{\MR}{\relax\ifhmode\unskip\space\fi MR }
\providecommand{\MRhref}[2]{%
  \href{http://www.ams.org/mathscinet-getitem?mr=#1}{#2}
}
\providecommand{\href}[2]{#2}

\end{document}